\documentclass[11pt]{article}
\usepackage[a4paper,margin=1.15in]{geometry}
\usepackage{amsmath,amssymb,amsthm,mathtools}
\usepackage{mathrsfs}
\usepackage{enumitem}
\usepackage{microtype}
\usepackage{xcolor}
\usepackage[colorlinks=true,linkcolor=blue!55!black,citecolor=blue!55!black,urlcolor=blue!55!black]{hyperref}

\newtheorem{theorem}{Theorem}[section]
\newtheorem{proposition}[theorem]{Proposition}
\newtheorem{lemma}[theorem]{Lemma}
\newtheorem{corollary}[theorem]{Corollary}
\theoremstyle{definition}
\newtheorem{definition}[theorem]{Definition}
\theoremstyle{remark}

\newcommand{\C}{\mathbb C}
\newcommand{\R}{\mathbb R}
\newcommand{\PP}{\mathbb P}
\newcommand{\eps}{\varepsilon}
\newcommand{\ii}{\sqrt{-1}}
\newcommand{\ddbar}{\ii\partial\bar\partial}
\newcommand{\Bl}{\operatorname{Bl}}
\newcommand{\Exc}{\operatorname{Exc}}
\newcommand{\Ric}{\operatorname{Ric}}
\newcommand{\Rm}{\operatorname{Rm}}
\newcommand{\sys}{\operatorname{sys}}
\newcommand{\tr}{\operatorname{tr}}
\newcommand{\dist}{\operatorname{dist}}

\newcommand{\supp}{\operatorname{supp}}
\newcommand{\Vol}{\operatorname{Vol}}
\newcommand{\ev}{\operatorname{ev}}

\newcommand{\norm}[1]{\left\lVert #1\right\rVert}
\newcommand{\abs}[1]{\left\lvert #1\right\rvert}
\title{Scalar curvature on K\"ahler blow-ups and systolic inequalities}
\author{Zehao Sha \and Jian Wang}
\date{}

\begin{document}
\maketitle

\begin{abstract}
In this paper, we develop the weighted level set method to a K\"ahler manifold $(X^n,\omega)$ admitting an almost holomorphic map to a possibly singular base $Z$, which is not uniruled. As a key intermediate result, we prove that any blowup $\operatorname{Bl}_SX$ of $X$ along smooth submanifolds $S$ of $ \operatorname{codim} S\ge2$ admits a sequence of K\"ahler metrics with scalar curvature globally and arbitrarily $C^0$-close to the scalar curvature of $\omega$. As a consequence, we establish the sharp \(2\)-systole estimate for every positive scalar curvature K\"ahler manifold $(X,\omega)$ and prove $\min_XS(\omega) \cdot\operatorname{sys}_2(\omega)  \le 2\pi r(r+1)$, where \(r\) is the rational dimension of $X$, with equality if and only if the universal cover splits as $(\widetilde X,\widetilde \omega) \cong (\mathbb P^r,\omega_{\mathrm{FS}}) \times(Y^{n-r},\omega_{\mathrm{RF}})$ up to normalization where $\omega_{\mathrm{FS}}$ is the Fubini-Study metric and $\omega_{\mathrm{RF}}$ is Ricci-flat. We also show a sharp even-systolic inequality in the same setting when the general fibre is the projective space.
\end{abstract}
\tableofcontents

\section{Introduction}

Systolic geometry studies the interaction between the topology of a manifold and the least volume of its nontrivial cycles. If $(M,g)$ is a closed Riemannian manifold, its homological $k$-systole is
\[
  \sys_k(g)
  :=
  \inf\bigl\{
    \Vol_g(Z):
    Z\text{ is an integral }k\text{-cycle and }
    [Z]\neq 0\in H_k(M;\mathbb Z)
  \bigr\}.
\]
Thus $\sys_k(M,g)$ measures the smallest $k$-dimensional scale at which the homology of $M$ can be detected. The systematic study of such quantities was initiated by Berger \cite{Ber72a,Ber72b} (see also \cite{Ber08}) and developed substantially by Gromov\cite{Gro83} through his work on filling radius, essential manifolds, and higher-dimensional systolic inequalities. A central problem is to determine when the systole can be controlled by global geometric quantities such as volume or curvature.

Among systolic invariants, the $2$-systole is the one most directly connected with scalar curvature. In dimension three, the stability inequality for an area-minimizing two-sphere, combined with the Gauss equation and the Gauss--Bonnet theorem, gives
\begin{equation}\label{3-dim-sys-est}
     \min_M\operatorname{scal}(g)\,
  \sys_2(g)\leq 8\pi.
\end{equation}
This mechanism underlies the sharp spherical $2$-systolic inequality of Bray--Brendle--Neves \cite{BBN10}, building on the minimal-surface theory of Schoen--Yau \cite{SY79}. Other methods reveal the same interaction from different directions. Kronheimer and Mrowka related scalar curvature to the Thurston norm by means of the three-dimensional Seiberg--Witten equations \cite{KM97}; Stern later recovered the sharp homological estimate \eqref{3-dim-sys-est} from the Bochner and coarea formulas for a harmonic map $M\to S^1$ \cite{Stern22}. Under additional topological hypotheses, such estimates persist in many other cases (see \cite{ZhuPAMS2020} and \cite{BrayBrendleEichmairNeves10}).

\vspace{2mm}

Such estimates do not hold for arbitrary positive-scalar-curvature manifolds. Indeed, on $S^2\times S^{n-2}$, $n\geq 5$, one may stretch the $S^2$-factor while keeping the second factor fixed. The scalar curvature remains uniformly positive, whereas the $2$-systole tends to infinity. Thus some additional geometric or topological structure is indispensable. The K\"ahler setting exhibits a markedly different behavior. Complex curves are calibrated by the K\"ahler form, and hence
\[
  \text{Area}_\omega(C)=\int_C\omega=[\omega]\cdot[C].
\]
Their areas are therefore determined cohomologically. Moreover, if a K\"ahler manifold \(X\) carries a PSC K\"ahler metric, the Chern--Weil identity shows that \(K_X\) is not pseudo-effective. By Ou's characterization \cite[Theorem 1.1]{Ou2025} of compact K\"ahler manifolds, \(X\) is uniruled. Hence \(X\) contains plenty of rational curves, which are topologically two-spheres. The K\"ahler structure supplies both canonical two-dimensional competitors and a rigid algebraic framework for controlling their areas.

The first sharp result in this direction was obtained by Sha \cite{Sha26Surface} for compact K\"ahler surfaces. With the convention that $S(\omega)$ denotes the Chern scalar curvature, so that $2S(\omega)=\operatorname{scal}(g_\omega)$, Sha proved
\[
  \min_X S(\omega)\,\sys_2(\omega)\leq 12\pi,
\]
with equality precisely for $\mathbb P^2$ endowed with the Fubini--Study metric, up to normalization. Using the birational classification of PSC K\"ahler surfaces, the optimal constant was further determined in each case: it is $12\pi$ for surfaces birational to $\mathbb P^2$, $8\pi$ for rational ruled surfaces, and $4\pi$ for non-rational ruled surfaces \cite[Theoren 1.2]{Sha26Surface}. In the last case, an independent analytic proof was obtained by adapting Stern's level-set method to the holomorphic ruling. An all-dimensional sharp estimate was subsequently established by Tsiamis \cite{Tsiamis26}.

These results suggest that the relevant constant should depend not only on the dimension of the manifold, but also on the dimension of its rationally connected part. The natural invariant measuring this contribution is provided by the maximal rationally connected (MRC) fibration. For a compact K\"ahler manifold $X$, the MRC fibration is an almost holomorphic map
\[
  \phi:X\dashrightarrow Z
\]
whose general fibre is rationally connected and whose base is not uniruled \cite{Campana81,KMM92,Kollar96}. It is unique up to bimeromorphic equivalence. We define the rational dimension of $X$ by
\[
  r(X):=\dim_{\mathbb C}F
  =\dim_{\mathbb C}X-\dim_{\mathbb C}Z,
\]
where $F$ is a general MRC fibre. In particular, $r(X)$ is a birational invariant, and $r(X)=\dim_{\mathbb C}X$ if and only if $X$ is rationally connected.

If \(X\) carries a PSC K\"ahler metric, $X$ is uniruled, and therefore $r(X)\geq 1$. Our first main theorem shows that this birational invariant controls the sharp $2$-systolic constant. 

\begin{theorem}\label{thm:2sys}
Let \((X^n,\omega)\) be a compact K\"ahler manifold with positive scalar curvature and let \(r=r(X)\) be its rational dimension. Then
\begin{equation}\label{eq:rational-dimensional-2sys-final}
        \min_XS(\omega)\,\sys_2(\omega)
        \le
        2\pi r(r+1).
\end{equation}
Moreover, equality holds if and only if the universal cover
\[
        (\widetilde X,\widetilde\omega)
        \cong
        (\PP^r,\sigma\omega_{\mathrm{FS}})
        \times(Y^{n-r},\omega_{\mathrm{RF}}),
        \qquad \sigma=\sys_2(\omega),
\]
where \(\omega_{\mathrm{RF}}\) is Ricci-flat and the \(2\)-systole is realized by a rational curve.  
\end{theorem}
When \(r(X)=n\), Theorem~\ref{thm:2sys} recovers the sharp dimensional bound given in \cite{Tsiamis26}
\[
  \min_X S(\omega)\,\sys_2(\omega)
  \leq 2\pi n(n+1),
\]
with equality precisely for projective space with its Fubini--Study metric. When $r<n$, the estimate in terms of $r$ is strictly stronger. Thus a scale-invariant differential-geometric quantity is bounded solely in terms of a birational invariant.

\vspace{2mm}

The proof has two principal ingredients. The first is a scalar-curvature approximation theorem for K\"ahler blow-ups. The MRC fibration is only almost holomorphic, and its base may be singular. Resolving the main component of its graph produces a diagram
\begin{equation}\label{eq:intro-mrc-resolution}
        \begin{array}{ccc}
        \widehat X & \xrightarrow{\widehat\pi} & \widehat Z \\
        \mu\downarrow & & \downarrow\nu \\
        X & \dashrightarrow & Z .
        \end{array}
\end{equation}
where $\widehat X$ and $\widehat Z$ are smooth compact K\"ahler manifolds, $\widehat\pi$ is holomorphic, and $\mu$ is a composition of blow-ups along smooth centres of complex codimension at least two \cite{BierstoneMilman97, Varouchas89}. Over a nonempty Zariski open subset of $\widehat Z$, the map $\widehat\pi$ is a proper holomorphic submersion, and its fibres are identified with the general MRC fibres of $X$.

There remains an analytic obstruction: the pullback metric \(\mu^*\omega\) is degenerate along the exceptional locus and hence cannot be used as a PSC K\"ahler metric on \(\widehat X\). We overcome this by replacing it
with a sequence of K\"ahler metrics whose scalar curvatures converge uniformly to \(\mu^*S(\omega)\).

Our construction is motivated by the broader theory of canonical metrics under blow-ups, initiated by Arezzo and Pacard for cscK metrics on point blow-ups \cite{ArezzoPacard06,ArezzoPacard09}. The theory was subsequently developed for extremal metrics by Arezzo--Pacard--Singer \cite{ArezzoPacardSinger11,S2012I,S2015II}, and for blow-ups along smooth centers of codimension greater than two by Seyyedali--Sz\'ekelyhidi \cite{SS20}. More recently, Boucksom--Jonsson--Trusiani developed a pluripotential-theoretic approach based on the openness of Mabuchi coercivity, extending the theory to weighted extremal metrics on equivariant resolutions of Fano type. In particular, their result applies to equivariant blow-ups along smooth centers of arbitrary codimension \cite{BoucksomJonssonTrusiani26}.

Our objective requires only uniform approximation of a prescribed scalar curvature. This weaker target allows us to start from an arbitrary compact K\"ahler metric and to blow up an arbitrary smooth center of complex codimension at least two, which generalizes the case of point blow-ups studied by Brown \cite{Brown24}.

\begin{theorem}
\label{thm:iterated-approximation}
Let $(X,\omega)$ be a compact K\"ahler manifold, and let \(X_J\xrightarrow{\mu_J}X_{J-1}\rightarrow\cdots\rightarrow X_1\xrightarrow{\mu_1}X_0=X\) be a finite sequence of nontrivial blow-ups along smooth complex submanifolds of codimension at least two. Denote by \(\mu:=\mu_1\circ\cdots\circ\mu_J:X_J\rightarrow X\). For $1\le j\le J$, let $E_j\subset X_j$ be the exceptional divisor of $\mu_j$, and let $\widetilde E_j$ denote its total transform on $X_J$.
Then there are parameter vectors \(\boldsymbol\eps_\ell=(\eps_{1,\ell},\ldots,\eps_{J,\ell})\rightarrow0\) and K\"ahler metrics $\omega_\ell$ on $X_J$ such that
\[
 [\omega_\ell]
 =\mu^*[\omega]-\sum_{j=1}^J\eps_{j,\ell}^2[\widetilde E_j],
\]
satisfying
\begin{equation}\label{eq:intro-iterated-scalar}
 \norm{S(\omega_\ell)-\mu^*S(\omega)}_{C^0(X_J)}
 \longrightarrow0 \qquad
 \omega_\ell\longrightarrow\mu^*\omega
\text{ in }C^\infty_{\mathrm{loc}}
 \bigl(X_J\setminus\operatorname{Exc}(\mu)\bigr).
\end{equation}
\end{theorem}

\vspace{2mm}

We now explain how Theorem~\ref{thm:iterated-approximation} completes the argument when \(r<n\). Applying it to the finite sequence of blow-ups defining \(\mu:\widehat X\to X\), we obtain K\"ahler metrics \(\widehat\omega_\ell\) on \(\widehat X\) such that $\min_{\widehat X}S(\widehat\omega_\ell)\rightarrow \min_XS(\omega)$, while their K\"ahler classes converge to \(\mu^*[\omega]\). We apply the weighted level-set argument to the resolved holomorphic fibration $\widehat\pi:\widehat X\longrightarrow\widehat Z$ equipped with \(\widehat\omega_\ell\), and obtain a fibrewise scalar curvature inequality \eqref{eq:psef-fibrewise-scalar-final}, then let \(\ell\to\infty\). Since the resolution can be chosen to be an isomorphism over the locus on which the MRC fibration is holomorphic and proper, the fibrewise rational curves obtained in this way descend to nontrivial rational curves on \(X\). Passing to the limit gives \eqref{eq:rational-dimensional-2sys-final} and hence proves Theorem~\ref{thm:2sys} in the remaining case \(r<n\). The rigidity statement is obtained by tracing the equality case through the fibrewise estimate. Equality forces the rationally connected fibre to be \(\PP^r\) with a multiple of the Fubini--Study metric and forces the transverse factor to be Ricci-flat. The universal Riemannian cover consequently splits as the product \((\PP^r,\omega_{\mathrm{FS}})\times (Y,\omega_{\mathrm{RF}})\).

The same method controls higher even-dimensional systoles when the general fibre is projective space.

\begin{theorem}\label{thm:even-sys}
Let \(X^{m+r}\) be a compact K\"ahler manifold admitting an almost holomorphic fibration \(f:X^{m+r}\dashrightarrow Z^m\) onto a normal compact K\"ahler space \(Z\) which is not uniruled. Assume
that a general fibre of \(f\) is biholomorphic to \(\PP^r\). If \(\omega\) is a K\"ahler metric on \(X\) with positive scalar curvature, then, for every \(1\le q\le r\),
\begin{equation}\label{eq:sharp-even-systolic}
        \min_X S(\omega)\cdot\sys_{2q}(\omega)^{1/q}
        \le 2\pi r(r+1).
\end{equation}
Moreover, for a fixed \(q\), equality holds if and only if the universal cover
\[
        (\widetilde X,\widetilde\omega)
        \cong
        (\PP^r,a\omega_{FS})\times(Y,\omega_{\mathrm{FS}}),
\]
where \(a=[\omega]\cdot[L]\) for a line \(L\) in a general fibre,
\(\omega_{\mathrm{RF}}\) is Ricci-flat, and \(\sys_{2q}(\omega)\) is realized by a
fibrewise linear \(\PP^q\subset\PP^r\).
\end{theorem}

In particular, the assumption on the general fibre is structural because it makes the fibrewise Chern--Weil identity and all even-dimensional systole depend on the same numerical parameter.  For a general rationally connected fibre, the minimal volumes of nontrivial \(2q\)-cycles may vary independently with the direction of the restricted K\"ahler class. For example, if the general fibre \(F\cong\PP^1\times\PP^1\), anisotropic product K\"ahler classes on \(F\cong\PP^1\times\PP^1\) show that \(S(\omega)\cdot\sys_4(\omega)^{1/2}\) is unbounded. Consequently, one can not expect uniform control for all even-systole.

\paragraph{Organization of the paper.}
In Section~\ref{sec:one-step}, we construct the approximate K\"ahler metrics on a blow-up along a smooth centre and establish the required weighted scalar-curvature estimates. 
In Section~\ref{sec:solving-equation}, we correct the approximation metrics by solving the nonlinear scalar-curvature equation and prove Theorem~\ref{thm:iterated-approximation}. 
In Section~\ref{sec:level-set-method}, we develop the weighted level-set method for holomorphic fibrations over bases with pseudo-effective canonical class.
Finally, in Section~\ref{sec:systolic-inequalities-psc}, we apply the weighted level-set method on graph resolutions of MRC fibrations. We first prove the even-dimensional systolic inequality, then establish the \(2\)-systole bound in terms of the rational dimension, and conclude with the analysis of the equality cases.

\paragraph{Note added.}
The initial draft of this paper was completed in May 2026, during the first author's visit to the second author at the CAS in Beijing. The project was subsequently set aside while both authors were occupied with other work. When the present manuscript was being prepared for circulation, we became aware of the independent work of Miao and Yang \cite{MiaoYang26}, which appeared in the arXiv listing of August 25, 2026. They independently prove a closely related scalar-curvature approximation theorem for blow-ups of compact K\"ahler manifolds of complex dimension at least three along smooth complex submanifolds of codimension at least two. The present work was developed without knowledge of their preprint.

\paragraph{Statement on the use of AI.}
All mathematical ideas, results, and the initial draft of this paper were developed independently by the authors. GPT-5.6-sol was subsequently used as an auxiliary tool for language polishing and for checking the internal consistency of mathematical proofs.

\paragraph{Acknowledgements.}
The first author thanks the second author for the warm hospitality during a visit in Beijing in May 2026. The authors thank Prof. Xiuxiong Chen, Song Sun, and Mingchen Xia for helpful discussions.

\section{Approximation metrics and scalar curvature estimates}
\label{sec:one-step}

Let $(X^n,\omega)$ be a compact K\"ahler manifold. The Ricci curvature \(\Ric(\omega)\) and the scalar curvature \(S(\omega)\) are defined by 
\[
\Ric(\omega) := -\ddbar \log \det (\omega), \qquad S(\omega):=\frac{n\Ric(\omega)\wedge \omega^{n-1}}{\omega^n}.
\]
Let $S\subset X$ be a smooth complex submanifold of complex codimension $k\ge2$, and write
\[
 \pi:\widetilde X:=\Bl_S X\longrightarrow X
\]
for the blowdown map, with exceptional divisor $E$. For every divisor \(D\), we use the normalization \(2\pi c_1\bigl(\mathcal O(D)\bigr)\in H^2(X;\R)\). In this section, we construct the approximation metrics, introduce the weighted H\"older spaces used in solving the scalar curvature equation, and estimate the scalar curvature of the approximation metrics. We shall use the following coordinates along submanifolds; see \cite[Lemma 5]{SS20}.
\begin{lemma}\label{pre:lem:coordinates}
Let \(d(x)=\dist_\omega(x,S)\). At every $p\in S$, there are holomorphic coordinates $(z_1,\ldots,z_k,w_1,\ldots,w_{n-k})$, centered at $p$, such that $S=\{z=0\}$,
\begin{equation}\label{pre:eq:distance-coordinate}
  d^2=|z|^2(1+\rho(z,w)),
  \qquad \rho=O(|z|+|w|),
\end{equation}
with all higher derivatives of $\rho$ uniformly bounded on smaller coordinate balls, and
\begin{equation}\label{pre:eq:potential-coordinate}
  \omega=\ddbar\bigl(|z|^2+|w|^2+\Phi(z,w)\bigr),
\end{equation}
where
\[
  \nabla^j\Phi=O(|z|^{3-j}+|w|^{3-j})\quad(0\le j<3),
\]
and all higher derivatives are uniformly bounded.
\end{lemma}

\subsection{Approximate metrics and H\"older spaces}
\label{sec:approximate}
Let $0<\eps\ll1$ be the gluing parameter. Choose a gluing scale \(r_\eps=\eps^\beta\), where $0<\beta<1$ is taken by
\[
  \beta=\frac{2k}{2k+1}\quad\text{if }k>2,
  \qquad
  \frac12<\beta<\frac23\quad\text{if }k=2.
\]
In particular, \(\eps\ll r_\eps\ll1\). We use the following four regions throughout:
\begin{equation}\label{gt:eq:regions}
\begin{array}{c|c}
\text{region} & \text{range of }d \\ \hline
\text{exterior} & d\ge2r_\eps \\
\text{outer annulus} & r_\eps\le d\le2r_\eps \\
\text{inner annulus} & 2\eps\le d\le r_\eps \\
\text{blowup region} & d\le2\eps.
\end{array}
\end{equation}

On the inner annulus, fix $x$ and write \(\rho=d(x),~ R=\rho/\eps\). Choose the adapted coordinates of Lemma~\ref{pre:lem:coordinates} at the nearest point of $S$ to $x$.  On a coordinate ball of radius $c\rho$ centred at $x$, where $c>0$ is fixed and sufficiently small, rescale by
\[
  z=\eps Z,\qquad w=\eps W.
\]
Then
\[
  |Z|\simeq R,\qquad |W|\le CR.
\]
All constants are independent of $\eps$, $R$, and the base point on $S$.

In the blowup region, use the standard affine charts of $\Bl_0\C^k$. On the chart where the $k$-th normal coordinate is distinguished, and after centring the tangential coordinates at the relevant base point, write
\begin{equation}\label{gt:eq:blowup-coordinates}
  z_k=\eps v,
  \qquad
  z_a=\eps vu_a\quad(1\le a\le k-1),
  \qquad
  w=\eps W.
\end{equation}
The exceptional divisor is given by $\{v=0\}$.  In these coordinates,
\[
  d^2=\eps^2|v|^2(1+|u|^2)(1+O(\eps)),
\]
with uniform estimates for all derivatives on compact subsets of the blowup chart. The other affine charts are obtained by permuting the normal coordinates.

Fix once and for all a smooth cutoff function $\gamma:[0,\infty)\to[0,1]$ such that
\[
  \gamma(t)=1\quad(t\le1),
  \qquad
  \gamma(t)=0\quad(t\ge2).
\]

\subsubsection{The Burns--Simanca scalar-flat metrics}
\label{subsubsec:burns-simanca-model}

For $k>2$, let $\eta_k$ be the Burns--Simanca scalar-flat K\"ahler metric on $\Bl_0\C^k$, constructed by Simanca \cite{Simanca1991}. On $\C^k\setminus\{0\}$, with coordinate $\zeta$, write
\begin{equation}\label{gt:eq:BS-potential}
  \eta_k
  =\ddbar\bigl(
     |\zeta|^2+\gamma(|\zeta|)\log|\zeta|^2
     +\psi_k(|\zeta|^2)
    \bigr).
\end{equation}
The radial remainder satisfies, for every $j\ge0$,
\begin{equation}\label{gt:eq:BS-potential-decay}
  \abs{\nabla_0^j\bigl(\psi_k(|\zeta|^2)\bigr)}
  \le C_j|\zeta|^{4-2k-j}
  \qquad(|\zeta|\ge2).
\end{equation}
Consequently,
\begin{equation}\label{gt:eq:BS-metric-decay}
  \abs{\nabla_0^j(\eta_k-\omega_0)}
  \le C_j|\zeta|^{2-2k-j},
  \qquad
  \abs{\nabla_0^j\Rm(\eta_k)}
  \le C_j|\zeta|^{-2k-j}.
\end{equation}
Here $\omega_0$ is the Euclidean metric.  The scalar-flat product model for \(k>2\) is
\begin{equation}\label{gt:eq:product-model}
  \omega_k=\eta_k+\omega_0
  \quad\text{on}\quad
  Y=\Bl_0\C^k\times\C^{n-k}.
\end{equation}

For $k=2$, let $\eta_2$ be the Burns scalar-flat K\"ahler metric on $\Bl_0\C^2$.  With our normalization, its potential on $\C^2\setminus\{0\}$ is
\begin{equation}\label{k2:eq:Burns-potential}
  \eta_2=\ddbar\bigl(|\zeta|^2+\log|\zeta|^2\bigr).
\end{equation}
In particular, \(\eta_2\) is asymptotically Euclidean and satisfies
\begin{equation}\label{k2:eq:Burns-decay}
  |\nabla_0^j(\eta_2-\omega_0)|\le C_j|\zeta|^{-2-j},
  \qquad
  |\nabla_0^j\Rm(\eta_2)|\le C_j|\zeta|^{-4-j}.
\end{equation}
The scalar-flat product model for \(k=2\) is
\begin{equation}\label{k2:eq:Burns-product}
  \omega_2=\eta_2+\omega_0
  \quad\text{on}\quad
  Y=\Bl_0\C^2\times\C^{n-2}.
\end{equation}

\subsubsection{The approximate metric in codimension greater than two}

Let \(\gamma_2=\gamma(r_\eps^{-1}d)\). On $X\setminus S\simeq\widetilde X\setminus E$, define
\begin{equation}\label{gt:eq:approximate-metric}
  \omega_\eps
  =\omega+\eps^2\ddbar\left[
    \gamma_2\left\{
      \gamma(\eps^{-1}d)\log(\eps^{-2}d^2)
      +\psi_k(\eps^{-2}d^2)
    \right\}
  \right].
\end{equation}
This is the approximate metric of \cite[Section~2.1]{SS20}. The next proposition combines \cite[Proposition~4]{SS20} with the regional estimates established in its proof. 

\begin{proposition}\label{gt:prop:approximate-metric}
Assume $k>2$.  For all sufficiently small $\eps>0$, the form $\omega_\eps$ defined by \eqref{gt:eq:approximate-metric} extends to a smooth K\"ahler metric on $\widetilde X$, and
\begin{equation}\label{gt:eq:class}
  [\omega_\eps]=\pi^*[\omega]-\eps^2[E].
\end{equation}
Moreover:
\begin{enumerate}[label=\textup{(\roman*)}]
\item on the exterior region, \(\omega_\eps=\omega\);
\item on the outer annulus, for every fixed $j\ge0$,
\begin{equation}\label{gt:eq:outer-potential-derivative}
  \abs{\nabla^j\left[
  \eps^2\gamma_2\psi_k(\eps^{-2}d^2)
  \right]}
  \le C_j\eps^{2k-2}r_\eps^{4-2k-j};
\end{equation}
\item on the inner annulus, potentials for $\eps^{-2}\omega_\eps$ and $\omega_k$ in the rescaled coordinates are respectively
\begin{align*}
  F&=|Z|^2+|W|^2
      +\eps^{-2}\Phi(\eps Z,\eps W)
      +\psi_k(\eps^{-2}d^2),\\
  F_{\mathrm{prod}}&=|Z|^2+|W|^2+\psi_k(|Z|^2),
\end{align*}
and, on the rescaled coordinate balls specified above,
\begin{equation}\label{gt:eq:inner-potential-error}
  \abs{\nabla_{Z,W}^j(F-F_{\mathrm{prod}})}
  \le C_j\eps^{j-2}d^{3-j}
  \qquad(2\le j\le6).
\end{equation}
In particular, $\eps^{-2}\omega_\eps$ is uniformly equivalent to $\omega_k$ there;
\item on the blowup region, in the coordinates
\eqref{gt:eq:blowup-coordinates},
\[
  \eps^{-2}\omega_\eps=\omega_k+O(\eps)
\]
in $C^\ell$ on compact subsets of each blowup chart, for every fixed $\ell$.  In particular, $\eps^{-2}\omega_\eps$ converges locally smoothly to $\omega_k$.
\end{enumerate}
\end{proposition}

\subsubsection{The approximate metric in codimension two}

For a K\"ahler metric $\Omega$, the linearization of the scalar curvature is defined by
\begin{equation}\label{pre:eq:linearization}
  L_\Omega\phi
  :=\left.\frac{d}{dt}\right|_{t=0}
    S(\Omega+t\ddbar\phi)
  =-\Delta_\Omega^2\phi
   -\langle\Ric(\Omega),\nabla^2_\Omega\phi\rangle_\Omega.
\end{equation}
Denote its formal $L^2(\Omega)$ adjoint by $L_\Omega^*$, and write
\[
  \mathcal K=\ker L_\omega,
  \qquad
  \mathcal K^*=\ker L_\omega^*.
\]
Since \(L_\omega:C^{4,\alpha}(X)\rightarrow C^{0,\alpha}(X)\) is elliptic of index zero,
\[
  \dim\mathcal K=\dim\mathcal K^*=:N.
\]
Fix an $L^2(\omega)$-orthonormal basis $\psi_1,\ldots,\psi_N$ of $\mathcal K^*$, and let $\delta_S$ be the current satisfying
\[
  \langle\delta_S,f\rangle=\int_S f\,(\omega|_S)^{n-2}.
\]

\begin{lemma}\label{k2:lem:Green-correction}
There are a real-valued function $\Gamma\in C^\infty(X\setminus S)$, a function $h_S\in\mathcal K^*$, and a nonzero constant $c_0$ depending only on the dimension such that, in the distributional sense,
\begin{equation}\label{k2:eq:Green-equation}
  L_\omega\Gamma=h_S-c_0\delta_S
\end{equation}
Moreover, in a tubular neighbourhood of $S$,
\begin{equation}\label{k2:eq:Green-expansion}
  \Gamma=\log d^2+B+\Psi,
\end{equation}
where $B\in C^\infty(S)$, extended smoothly to the neighbourhood, and
\begin{equation}\label{k2:eq:Psi-estimate}
  |\nabla^j\Psi|\le C_jd^{1-j}
  \qquad(j\ge0).
\end{equation}
\end{lemma}

\begin{proof}
Define the projection \(\Pi^*:\mathcal D'(X)\longrightarrow \ker L_\omega^*\) by
\[
  \Pi^*T=\sum_{j=1}^N\langle T,\psi_j\rangle\psi_j.
\]
Since $L_\omega$ is elliptic of index zero, it has a Green operator $G$ such that
\[
  L_\omega GT=T-\Pi^*T.
\]
Set $\Gamma=-c_0G\delta_S$ for some constant \(c_0\). Then
\[
  L_\omega\Gamma=-c_0\delta_S+c_0\Pi^*\delta_S,
\]
which gives \eqref{k2:eq:Green-equation} by arranging \(h_S=c_0\Pi^*\delta_S\).

It remains to choose $c_0$ and determine the singularity of $\Gamma$. The Green kernel of a fourth-order elliptic operator has the same leading singularity as the fundamental solution of its principal part (see
\cite[Chapter~XVIII, Sections~18.1--18.2]{HormanderIII}). Since the principal part of $L_\omega$ is $-\Delta_\omega^2$, this leading term is a nonzero multiple of $|x-y|^{4-2n}$ when $n>2$.

Fix a point of $S$ and use normal coordinates there.  Integrating the leading term in the $2n-4$ tangential variables gives
\[
 \begin{aligned}
  \int_{|u|<1}\frac{du}{(d^2+|u|^2)^{n-2}}
  &=|\mathbb S^{2n-5}|
    \int_0^1\frac{r^{2n-5}}{(d^2+r^2)^{n-2}}\,dr\\
  &=-\frac{|\mathbb S^{2n-5}|}{2}\log d^2+O(1)
  \qquad(d\to0).
 \end{aligned}
\]
Thus real codimension four is precisely the case in which the Green
potential of $\delta_S$ has a logarithmic singularity.  When $n=2$, the
submanifold $S$ is zero-dimensional and the fundamental solution of the
four-dimensional bi-Laplacian is itself a nonzero multiple of $\log d^2$.
In both cases the coefficient is universal and nonzero.  We choose $c_0$
so that the coefficient of $\log d^2$ in $-c_0G\delta_S$ is one.

The differentiated local expansion of the same parametrix then gives a
smooth function $B$ on $S$ such that
\[
  \Gamma-\log d^2-B=\Psi,
  \qquad
  |\nabla^j\Psi|\le C_jd^{1-j}\quad(j\ge0).
\]
This proves \eqref{k2:eq:Green-expansion}--\eqref{k2:eq:Psi-estimate}.
Elliptic regularity away from $S$ gives
$\Gamma\in C^\infty(X\setminus S)$.
\end{proof}

Let \(\gamma_1:= 1- \gamma_2\), where \(\gamma_2=\gamma(r_\eps^{-1}d)\). After extending $B$ smoothly away from $S$, on $X\setminus S$ we define
\begin{align}
\omega_\eps=\omega+\eps^2\ddbar A_\eps, \qquad  A_\eps
  :=\gamma_2\log(\eps^{-2}d^2)
    +\gamma_1(\Gamma-2\log\eps-B)+B.
  \label{k2:eq:approximate-metric}
\end{align}
Using \eqref{k2:eq:Green-expansion}, the potential on the outer annulus can also be written as
\begin{equation}\label{k2:eq:Aeps-annulus}
  A_\eps=\log(\eps^{-2}d^2)+B+\gamma_1\Psi.
\end{equation}

\begin{proposition}\label{k2:prop:approximate-metric}
For all sufficiently small $\eps>0$, $\omega_\eps$ extends to a smooth K\"ahler metric on $\widetilde X$, and
\begin{equation}\label{k2:eq:class}
  [\omega_\eps]=\pi^*[\omega]-\eps^2[E].
\end{equation}
\end{proposition}

\begin{proof}
We work on four regions \eqref{gt:eq:regions}. If $\{d\ge2r_\eps\}$, then $A_\eps=\Gamma-2\log\eps$, so
\begin{equation}\label{k2:eq:exterior-approximate-metric}
  \omega_\eps=\omega+\eps^2\ddbar\Gamma.
\end{equation}
Lemma \ref{k2:lem:Green-correction} gives $|\nabla^2\Gamma|\le Cd^{-2}$ near $S$, while $\Gamma$ is smooth away from $S$. Hence the metric perturbation is at most
\[
  C\eps^2r_\eps^{-2}=C\eps^{2-2\beta}=o(1).
\]

On $\{r_\eps\le d\le2r_\eps\}$, use \eqref{k2:eq:Aeps-annulus}. Equations \eqref{k2:eq:Psi-estimate} and the derivative bounds of cutoff functions give
\[
  |\nabla^2 A_\eps|
  \le C(r_\eps^{-2}+r_\eps^{-1}+1).
\]
Thus $\eps^2\ddbar A_\eps=o(1)$ relative to $\omega$.

On $\{2\eps\le d\le r_\eps\}$, one has $A_\eps=\log(\eps^{-2}d^2)+B$. Using the coordinates of Lemma~\ref{pre:lem:coordinates} at the nearest point of $S$ and writing $(z,w)=\eps(Z,W)$, potentials for $\eps^{-2}\omega_\eps$ and the product Burns metric are
\begin{align*}
  F&=|Z|^2+|W|^2+\log|Z|^2+H_\eps,\\
  F_2&=|Z|^2+|W|^2+\log|Z|^2,
\end{align*}
where
\begin{equation}\label{k2:eq:Heps}
  H_\eps
  =\eps^{-2}\Phi(\eps Z,\eps W)
   +\log(\eps^{-2}d^2)-\log|Z|^2
   +B(\eps Z,\eps W).
\end{equation}
On a rescaled ball centered where $d\simeq\eps R$, we have
\begin{align*}
  |\nabla^2\eps^{-2}\Phi(\eps Z,\eps W)|&\le Cd,\\
  |\nabla^2(\log(\eps^{-2}d^2)-\log|Z|^2)|&\le C\eps^2d^{-1},\\
  |\nabla^2B(\eps Z,\eps W)|&\le C\eps^2.
\end{align*}
All three quantities tend uniformly to zero, so the scaled metric is a small perturbation of $\omega_2$.

Finally, consider the blowup region $\{d\le2\eps\}$.  Here $\gamma_2=1$ for small $\eps$, so \(A_\eps=\log(\eps^{-2}d^2)+B\). Fix $p\in S$. On the blowup chart \eqref{gt:eq:blowup-coordinates} over $p$, we have $Z=(vu,v)$ and
\[
  |Z|^2=|v|^2(1+|u|^2).
\]
By \eqref{pre:eq:distance-coordinate},
\begin{align*}
 \log(\eps^{-2}d^2)
 =\log|v|^2+\log(1+|u|^2)
 +\log\bigl(1+\rho(\eps(vu,v),\eps W)\bigr).
\end{align*}
Consequently, a potential for $\eps^{-2}\omega_\eps$ on the chart is
\begin{align*}
 F_\eps={}&|v|^2(1+|u|^2)+|W|^2
   +\log|v|^2+\log(1+|u|^2)\\
 &+\eps^{-2}\Phi(\eps(vu,v),\eps W)
   +\log\bigl(1+\rho(\eps(vu,v),\eps W)\bigr)\\
 &+B(\eps(vu,v),\eps W).
\end{align*}
The function $\log|v|^2$ is plurisubharmonic with a logarithmic singularity along $\{v=0\}$, and it is pluriharmonic on $\{v\ne0\}$. Thus it does not contribute to $\ddbar F_\eps$ on the chart. After removing this term and the constant $B(p)$, the remaining potential is smooth across $v=0$.

Moreover, on every fixed compact subset of the chart,
\begin{align*}
 \eps^{-2}\Phi(\eps(vu,v),\eps W)&=O(\eps),\\
 \log\bigl(1+\rho(\eps(vu,v),\eps W)\bigr)&=O(\eps),\\
 B(\eps(vu,v),\eps W)-B(p)&=O(\eps)
\end{align*}
in every fixed $C^\ell$ norm.  Hence this smooth local potential is an
$O(\eps)$ perturbation of
\[
  |v|^2(1+|u|^2)+|W|^2+\log(1+|u|^2),
\]
which is the product Burns model potential in this chart. Indeed, $u$ is the affine coordinate on the $\PP^1$-fibre of $E=\PP(N_{S/X})$, and $\ddbar\log(1+|u|^2)$ is the Fubini--Study form on that fibre.  It follows that $\eps^{-2}\omega_\eps$ extends smoothly across $E$ and is positive there for small $\eps$.

Together with the estimates in the other three regions, this proves that $\omega_\eps$ is K\"ahler for small $\eps$. Its background term has class $\pi^*[\omega]$.  On overlaps of the blowup charts, the local logarithmic potentials obtained after removing $\log|v|^2$ differ by the logarithm of the squared modulus of a holomorphic transition function of $\mathcal O(-E)$.  Thus the extended logarithmic correction represents $-[E]$ in our normalization.  All remaining correction terms are globally $\ddbar$-exact.  Since the logarithmic correction has coefficient $\eps^2$, we obtain \eqref{k2:eq:class}.
\end{proof}

\subsubsection{Weighted H\"older spaces}

We denote by the same symbol \(d\) the pullback of the distance function to \(\widetilde X\setminus E\), and extend it continuously by \(d=0\) on \(E\). The metric $\omega_\eps$ has three natural length scales depending on the distance to \(E\) which is measured by the weight function
\begin{equation}
\label{eq:weight}
  \tau_\eps(x)=
  \begin{cases}
    1,&d(x)\ge1,\\
    d(x),&\eps\le d(x)\le1,\\
    \eps,&d(x)\le\eps.
  \end{cases}
\end{equation}
We usually write simply \(\tau\) when \(\eps\) is fixed.  Near the exceptional divisor, \(\tau\simeq\eps\), in the gluing annulus, \(\tau\simeq d\), and away from \(S\), \(\tau\simeq1\).

\begin{definition}[Weighted H\"older space]
\label{def:weighted-holder}
For $\ell\ge0$, $0<\alpha<1$, and $\mu\in\R$, define
\begin{equation}
\label{eq:weighted-norm}
  \norm{f}_{C^{\ell,\alpha}_\mu(\widetilde X)}
  =\sup_{x\in\widetilde X}\tau(x)^{-\mu}
   \norm{f}_{C^{\ell,\alpha}(B_x,\tau(x)^{-2}\omega_\eps)},
\end{equation}
where \(B_x=B_{\omega_\eps}(x,c\tau(x))\) for a fixed and sufficiently small constant $c>0$.  The local H\"older norm is computed using the rescaled metric $\tau(x)^{-2}\omega_\eps$. In particular, $B_x$ has uniformly bounded radius in this rescaled metric. For this choice of $c$, one also has $\tau(y)\simeq\tau(x)$ for every $y\in B_x$, uniformly in $x$ and small $\eps$. Changing $c$, the balls, or the uniformly equivalent smooth representative of $\tau$ changes the norm only by a constant independent of small $\eps$. The resulting Banach space is denoted by $C^{\ell,\alpha}_\mu(\widetilde X)$.

The factor $\tau(x)^{-\mu}$ records the size of the function, while the metric rescaling records the size of its derivatives.  In particular,
\begin{equation}
\label{eq:weighted-derivative}
  \norm{f}_{C^{\ell,\alpha}_\mu}\le C
  \quad\Longrightarrow\quad
  |\nabla^j f|_{\omega_\eps}\le C'\tau^{\mu-j}
  \quad(0\le j\le\ell),
\end{equation}
with the corresponding scaled H\"older bounds.
\end{definition}

We next spell out the coordinate meaning of the definition.  On $\{d\ge r_\eps\}$, the approximate metric is uniformly equivalent to the background metric on balls of radius comparable to $\tau$. Hence \eqref{eq:weighted-derivative} gives
\[
  |\nabla_\omega^j f|\le C'\tau^{\mu-j}
  \qquad(0\le j\le\ell).
\]
In particular, the right-hand side is $C'd^{\mu-j}$ where $r_\eps\le d<1$, and it is uniformly bounded where $d\ge1$.

On $\{2\eps\le d\le r_\eps\}$, fix $x$, put $\rho=d(x)$ and $R=\rho/\eps$, and use the coordinates $(z,w)=\eps(Z,W)$ introduced above. On the corresponding
rescaled ball,
\[
  \tau\simeq\rho=\eps R,
  \qquad
  \tau^{-2}\omega_\eps
  =R^{-2}(\eps^{-2}\omega_\eps).
\]
Since $\eps^{-2}\omega_\eps$ is uniformly equivalent to the relevant scalar-flat product metric, a uniform weighted bound implies
\begin{equation}\label{eq:weighted-inner-coordinate}
  |\nabla_{Z,W}^j f|
  \le C'\eps^\mu R^{\mu-j}
  \qquad(0\le j\le\ell),
\end{equation}
where the derivatives are measured with respect to that product model.

On $\{d\le2\eps\}$, use the coordinates $(u,v,W)$ from \eqref{gt:eq:blowup-coordinates}. Here $\tau\simeq\eps$ and $\eps^{-2}\omega_\eps$ is uniformly equivalent to the corresponding product model.  Therefore
\begin{equation}\label{eq:weighted-blowup-coordinate}
  |\nabla_{u,v,W}^j f|
  \le C'\eps^\mu
  \qquad(0\le j\le\ell).
\end{equation}
Conversely, the preceding estimates, together with the corresponding H\"older estimates on the same rescaled charts, give a uniform bound in $C^{\ell,\alpha}_\mu(\widetilde X)$.

We shall repeatedly use the following elementary consequence of the definition.  If \(h_\eps\) is uniformly bounded in \(C^{0,\alpha}_0\) and
\[
  \supp h_\eps\subset\{d\le R_\eps\},
  \qquad
  \eps\le R_\eps\le1,
\]
then, for every \(\delta<4\),
\begin{equation}\label{eq:support-gain}
  \norm{h_\eps}_{C^{0,\alpha}_{\delta-4}}
  \le C R_\eps^{4-\delta}.
\end{equation}
If the ball \(B_x\) meets \(\supp h_\eps\), then \(\tau(x)\le C R_\eps\);
otherwise the local norm is zero. Hence
\[
  |h_\eps|
  \le C
  \le C R_\eps^{4-\delta}\tau^{\delta-4}.
\]
The H\"older semi-norm is estimated in the same rescaled balls. The weight range depends on the codimension. We choose $\delta \in (4-2k,0)$ for \(k>2\) and $\delta \in (-1,0)$ for $k=2$.

\subsection{The scalar curvature estimate}
\label{sec:scalar-error}

The goal of this subsection is the following weighted scalar curvature estimate.

\begin{proposition}
\label{prop:scalar-error}
The following estimates hold uniformly for small $\eps$.
\begin{enumerate}[label=\textup{(\roman*)}]
\item If $k>2$ and $4-2k<\delta<0$, then
\begin{equation}\label{gt:eq:scalar-error-estimate}
  \norm{S(\omega_\eps)-\pi^*S(\omega)}
  _{C^{0,\alpha}_{\delta-4}}
  \le C\Theta_\eps,
\end{equation}
where \(\Theta_\varepsilon=r_\varepsilon^{3-\delta}+\varepsilon^{2k-2}r_\varepsilon^{4-\delta-2k}\). Moreover,
\begin{equation}\label{gt:eq:Theta-limits}
  \Theta_\varepsilon\to0
  \qquad\text{and}\qquad
  \varepsilon^{\delta-2}\Theta_\varepsilon\to0
  \qquad\text{as }\varepsilon\to0.
\end{equation}
\item If $k=2$ and $-1<\delta<0$, then
\begin{equation}\label{k2:eq:scalar-error}
  \norm{S(\omega_\varepsilon)
  -
  \pi^*S(\omega)
  -
  \varepsilon^2 h_S}_{C^{2,\alpha}_{\delta-4}}
  \le Cr_\eps^{4-\delta},
\end{equation}
where \(h_S\) is defined as in Lemma \ref{k2:lem:Green-correction}. Moreover,
\begin{equation}\label{k2:eq:error-limits}
  r_\varepsilon^{4-\delta}\to0
  \qquad\text{and}\qquad
  \varepsilon^\delta r_\varepsilon^{4-\delta}\to0
  \qquad\text{as }\varepsilon\to0.
\end{equation}
\end{enumerate}
\end{proposition}

\subsubsection{Estimates in codimension \(k>2\)}

We first deal with the case when \(k>2\).

\begin{proof}[Proof of Proposition~\ref{prop:scalar-error} \textup{(i)}]
We work in the four regions of \eqref{gt:eq:regions}. On $\{d\ge2r_\eps\}$, the two metrics agree and the resulting estimate \eqref{gt:eq:scalar-error-estimate} is obvious. On $\{d\le2\eps\}$, Proposition~\ref{gt:prop:approximate-metric} \textup{(iv)} gives
\[
  \eps^{-2}\omega_\eps=\omega_k+O(\eps)
\]
in the standard $C^{4,\alpha}$ norm on each fixed blow-up chart. Since
$\omega_k$ is scalar-flat, we obtain
\[
  \norm{S(\eps^{-2}\omega_\eps)}_{C^{0,\alpha}}\le C\eps.
\]
Since $\tau\simeq\eps$ on this region and $S(\omega)$ is bounded, Definition~\ref{def:weighted-holder} gives
\begin{equation}\label{gt:eq:blowup-scalar-error}
  \norm{S(\omega_\eps)-S(\omega)}
  _{C^{0,\alpha}_{\delta-4}(\{d\le2\eps\})}
  \le C(\eps^{3-\delta}+\eps^{4-\delta})
  \le Cr_\eps^{3-\delta}.
\end{equation}

On the inner annulus $\{2\eps\le d\le r_\eps\}$, fix a point with $d=\rho$ and let $R=\rho/\eps$.  Use the scaled potentials in Proposition~\ref{gt:prop:approximate-metric} \textup{(iii)}. It follows from \eqref{gt:eq:inner-potential-error} that,
\[
  \nabla^2(F-F_{\mathrm{prod}})=O(\rho),
  \qquad
  \nabla^4(F-F_{\mathrm{prod}})=O(\eps^2\rho^{-1}).
\]
The second-derivative bound makes the metric perturbation uniformly small. The bounds in \eqref{gt:eq:inner-potential-error} through order six, together with \eqref{gt:eq:BS-metric-decay}, show on a coordinate ball of radius comparable to $R$ that
\[
  \norm{S(\eps^{-2}\omega_\eps)}
  _{C^{0,\alpha}(B_x,R^{-2}\eps^{-2}\omega_\eps)}
  \le C\eps^2\rho^{-1}.
\]
Indeed, the linear term is controlled by the fourth derivatives of $F-F_{\mathrm{prod}}$, while every nonlinear term contains either a second and a fourth derivative or two third derivatives.  Since $R^{-2}\eps^{-2}\omega_\eps=\rho^{-2}\omega_\eps$ and scalar curvature scales by the inverse square of length, this gives
\[
  \norm{S(\omega_\eps)}
  _{C^{0,\alpha}(B_x,\rho^{-2}\omega_\eps)}
  \le C\rho^{-1}.
\]
Since $\tau\simeq d$,
\begin{equation}\label{gt:eq:inner-scalar-error}
  \norm{S(\omega_\eps)-S(\omega)}
  _{C^{0,\alpha}_{\delta-4}(\{2\eps\le d\le r_\eps\})}
  \le C(r_\eps^{3-\delta}+r_\eps^{4-\delta})
  \le Cr_\eps^{3-\delta}.
\end{equation}

On the outer annulus \(\{r_\eps\le d\le 2r_\eps\}\), write \(u_\eps=\eps^2\gamma_2\psi_k(\eps^{-2}d^2)\).
Equation~\eqref{gt:eq:outer-potential-derivative} gives, for $0\le j\le6$,
\begin{equation}\label{gt:eq:outer-uj}
  |\nabla^j u_\eps|
  \le C\eps^{2k-2}r_\eps^{4-2k-j}.
\end{equation}
Expanding at $\omega$,
\[
  S(\omega+\ddbar u_\eps)-S(\omega)
  =L_\omega u_\eps+Q_\omega(u_\eps).
\]
Here $Q_\omega(u_\eps)$ denotes the nonlinear remainder. The fourth-order linear term is at most $C\eps^{2k-2}r_\eps^{-2k}$. The metric perturbation has size
\[
  \eta_\eps:=\eps^{2k-2}r_\eps^{2-2k}
  =(\eps/r_\eps)^{2k-2}=o(1).
\]
Using \eqref{gt:eq:outer-uj}, every term in $Q_\omega$ is bounded by
$C\eta_\eps\eps^{2k-2}r_\eps^{-2k}$ and hence by the same linear bound.
Thus, for every $x$ in the outer annulus,
\[
  \norm{S(\omega_\eps)-S(\omega)}
  _{C^{0,\alpha}(B_x,d(x)^{-2}\omega_\eps)}
  \le C\eps^{2k-2}r_\eps^{-2k}.
\]
Since $d\simeq r_\eps$ on this region,
\begin{equation}\label{gt:eq:outer-scalar-error}
  \norm{S(\omega_\eps)-S(\omega)}
  _{C^{0,\alpha}_{\delta-4}}
  \le C\eps^{2k-2}r_\eps^{4-\delta-2k}.
\end{equation}
Combining \eqref{gt:eq:blowup-scalar-error},
\eqref{gt:eq:inner-scalar-error}, and \eqref{gt:eq:outer-scalar-error} proves
\eqref{gt:eq:scalar-error-estimate}.

It remains to verify \eqref{gt:eq:Theta-limits}. Write \(\delta=4-2k+\nu\) for some \(0<\nu<2k-4\). Then
\[
  r_\eps^{3-\delta}=\eps^{\beta(3-\delta)}\to0,
  \qquad
  \eps^{2k-2}r_\eps^{4-\delta-2k}
  =\eps^{2k-2-\beta\nu}\to0.
\]
The last exponent is positive because
$0<\nu<2k-4$ and $0<\beta<1$.  Furthermore,
\begin{align*}
 \eps^{\delta-2} \Theta_\eps=\eps^{(2+\nu)/(2k+1)}+\eps^{\nu(1-\beta)}\to0.
\end{align*}
This completes the proof.
\end{proof}

\subsubsection{Estimates in codimension \(k=2\)}

To prove Proposition \ref{prop:scalar-error} (ii), we need the following refined scalar curvature estimate on the inner annulus.

\begin{lemma}\label{k2:lem:refined-inner}
Let $x$ satisfy $2\eps\le d(x)=\rho\le r_\eps$, and let $B_x$ be a ball of radius comparable to $\rho$ as in \eqref{eq:weighted-norm}. Then
\begin{equation}\label{k2:eq:refined-inner}
  \norm{S(\omega_\eps)}
  _{C^{2,\alpha}(B_x,\rho^{-2}\omega_\eps)}
  \le C\left(1+\eps^2\rho^{-3}\right).
\end{equation}
\end{lemma}
\begin{proof}
Fix a point \(x\) with $d(x)=\rho$ and let $R=\rho/\eps$. Use the scaled coordinates in \eqref{k2:eq:Heps}. Taking the Taylor expansion of the background potential at the nearest point $p\in S$ gives, with derivative control through order six,
\begin{equation}\label{k2:eq:Phi-Taylor}
  \eps^{-2}\Phi(\eps Z,\eps W)
  =\eps P_{3,p}(Z,W)+\eps^2P_{4,p}(Z,W)
   +\mathcal R_{\eps,p}(Z,W),
\end{equation}
where $P_{j,p}$ is homogeneous of degree $j$ and, on the relevant
coordinate ball,
\begin{equation}\label{k2:eq:Phi-remainder}
  |\nabla^j\mathcal R_{\eps,p}|
  \le C\eps^3R^{5-j}
  \qquad(0\le j\le6).
\end{equation}
Since $P_{3,p}$ has degree three,
\[
  \Delta_0^2P_{3,p}=0.
\]
Let $L_{\omega_2}$ be the scalar-curvature linearization of the product model metric. From \eqref{k2:eq:Burns-decay}, the coefficient difference between $L_{\omega_2}$ and \(-\Delta_0^2\) has the decay
\begin{equation}\label{k2:eq:cubic-Burns}
  |L_{\omega_2}P_{3,p}|\le CR^{-3}.
\end{equation}
Indeed, fourth derivatives of $P_{3,p}$ vanish, while the remaining terms contain either a decaying derivative of the metric times a derivative of order at most three, or the $O(R^{-4})$ Ricci tensor times
$\nabla^2P_{3,p}=O(R)$.

It follows from \eqref{pre:eq:distance-coordinate} that, for $0\le j\le6$,
\begin{equation}\label{k2:eq:log-distance-derivatives}
  \left|\nabla^j\bigl(
  \log(\eps^{-2}d^2)-\log|Z|^2\bigr)\right|
  \le C\eps R^{1-j}.
\end{equation}
The quartic term, \eqref{k2:eq:Phi-remainder}, and the smooth function $B$ satisfy
\begin{equation}\label{k2:eq:inner-linear-pieces}
  |L_{\omega_2}(\eps^2P_{4,p})|
  +|L_{\omega_2}\mathcal R_{\eps,p}|
  +|L_{\omega_2}B(\eps Z,\eps W)|
  \le C\eps^2;
\end{equation}
for the remainder we use \( \varepsilon R=d\le r_\varepsilon\to0\). Equations \eqref{k2:eq:cubic-Burns}--\eqref{k2:eq:inner-linear-pieces} imply
\begin{equation}\label{k2:eq:LBH}
  |L_{\omega_2}H_\eps|\le C(\eps R^{-3}+\eps^2).
\end{equation}
Moreover, the relevant derivatives of $H_\eps$ satisfy
\begin{align*}
  |\nabla^2H_\eps|&\le
  C(\eps R+\eps R^{-1}+\eps^2R^2),\\
  |\nabla^3H_\eps|&\le
  C(\eps+\eps R^{-2}+\eps^2R),\\
  |\nabla^4H_\eps|&\le
  C(\eps R^{-3}+\eps^2).
\end{align*}
Since $R\ge2$ and $\eps R\le r_\eps=o(1)$, the nonlinear remainder \(Q(H_\varepsilon)=S(\omega_2+\sqrt{-1}\partial\bar\partial H_\varepsilon)-L_{\omega_2}H_\varepsilon\) is bounded by
\[
  C\bigl(|\nabla^2H_\eps||\nabla^4H_\eps|
        +|\nabla^3H_\eps|^2\bigr)
  \le C(\eps R^{-3}+\eps^2).
\]
Differentiating the scalar-curvature expansion at most twice and using the bounds through order six gives
\[
  \norm{S(\eps^{-2}\omega_\eps)}
  _{C^{2,\alpha}(B_x,R^{-2}\eps^{-2}\omega_\eps)}
  \le C(\eps R^{-3}+\eps^2).
\]
Since the product model metric is scalar-flat, scaling back to $\omega_\eps$ and using $\rho=\eps R$ gives
\[
  \norm{S(\omega_\eps)}
  _{C^{2,\alpha}(B_x,\rho^{-2}\omega_\eps)}
  \le C\eps^{-2}(\eps R^{-3}+\eps^2)
  =C(1+\eps^2\rho^{-3}),
\]
which proves the lemma.
\end{proof}

We now prove the case \(k=2\) for Proposition \ref{prop:scalar-error}. For convenience, write \(E_\eps=S(\omega_\eps)-\pi^*S(\omega)-\eps^2h_S\).

\begin{proof}[Proof of Proposition~\ref{prop:scalar-error} \textup{(ii)}]
On $\{d\ge2r_\eps\}$, $\omega_\eps=\omega+\eps^2\ddbar\Gamma$. Since
$L_\omega\Gamma=h_S$ away from $S$,
\[
  E_\eps=Q_\omega(\eps^2\Gamma).
\]
By Lemma \ref{k2:lem:Green-correction}, the bounds $|\nabla^j\Gamma|\le C_jd^{-j}$ for $j\ge1$ give, for
$0\le j\le2$,
\[
  |\nabla^jE_\eps|\le C\eps^4d^{-6-j}.
\]
Multiplication by the weighted factors gives
\begin{equation}\label{k2:eq:exterior-error}
  \norm{E_\eps}_{C^{2,\alpha}_{\delta-4}(\{d\ge2r_\eps\})}
  \le C\eps^4r_\eps^{-2-\delta}
  \le Cr_\eps^{4-\delta},
\end{equation}
because $\eps^4\le r_\eps^6$ when $\beta<2/3$.

On $\{r_\eps\le d\le2r_\eps\}$, equation \eqref{k2:eq:Aeps-annulus} gives
\[
  \omega_\eps
  =\omega+\eps^2\ddbar\Gamma
   -\eps^2\ddbar(\gamma_2\Psi).
\]
Since $|\nabla^j(\gamma_2\Psi)|\le Cr_\eps^{1-j}$,
\[
  \eps^2|\nabla^jL_\omega(\gamma_2\Psi)|
  \le C\eps^2r_\eps^{-3-j}
  \quad(0\le j\le2).
\]
The nonlinear remainder is at most $C\eps^4r_\eps^{-6-j}$.  Therefore, using $\eps^2\le r_\eps^3$, we have
\begin{equation}\label{k2:eq:outer-error}
  \norm{E_\eps}_{C^{2,\alpha}_{\delta-4}}
  \le C\left(
    \eps^2r_\eps^{1-\delta}
    +\eps^4r_\eps^{-2-\delta}
  \right)
  \le Cr_\eps^{4-\delta}.
\end{equation}

On $\{2\eps\le d\le r_\eps\}$, Lemma~\ref{k2:lem:refined-inner} and the boundedness of $S(\omega)$ and $h_S$, together with Definition~\ref{def:weighted-holder}, give
\[
  \norm{E_\eps}
  _{C^{2,\alpha}_{\delta-4}(\{2\eps\le d\le r_\eps\})}
  \le C\sup_{2\eps\le d\le r_\eps}
  \left(d^{4-\delta}+\eps^2d^{1-\delta}\right)
  \le Cr_\eps^{4-\delta}.
\]
Here the second term is bounded by the right-hand side because $\eps^2\le r_\eps^3$.

On $\{d\le2\eps\}$, Proposition~\ref{k2:prop:approximate-metric} gives $\eps^{-2}\omega_\eps=\omega_2+O(\eps)$ in the standard $C^{6,\alpha}$ norm on each fixed blow-up chart.  Since $\omega_2$ is scalar-flat and $\tau\simeq\eps$, we have
\begin{equation}\label{k2:eq:blowup-error}
  \norm{E_\eps}_{C^{2,\alpha}_{\delta-4}(\{d\le2\eps\})}
  \le C\eps^{3-\delta}
  \le Cr_\eps^{4-\delta}.
\end{equation}
The last inequality follows from $3-\delta>\frac23(4-\delta)>\beta(4-\delta)$ for $-1<\delta<0$. Combining the four regions proves \eqref{k2:eq:scalar-error}.

Finally,
\[
  r_\eps^{4-\delta}=\eps^{\beta(4-\delta)}\to0,
\]
and
\[
  \eps^\delta r_\eps^{4-\delta}
  =\eps^{4\beta+(1-\beta)\delta}\to0,
\]
because $4\beta+(1-\beta)\delta>5\beta-1>0$.  This proves
\eqref{k2:eq:error-limits}.
\end{proof}

\section{Solving the scalar curvature equation}
\label{sec:solving-equation}

The goal of this section is to prove Theorem \ref{thm:iterated-approximation}. We first solve the linearized equation in the weighted H\"older spaces.

\subsection{The linear estimate}
\label{sec:linear}

The following construction comes from \cite[Lemma~4.1]{Brown24}, applied to $\mathcal K=\ker L_\omega$ rather than to \(\mathcal K^*\). The same argument applies because restriction from $\mathcal K$ to the dense open set $X\setminus S$ is injective.
\begin{lemma}\label{lem:evaluation-points}
There are points $q_1,\ldots,q_N\in X\setminus S$ such that
\begin{equation}\label{eq:evaluation}
  \ev_q:\mathcal K\longrightarrow\R^N,
  \qquad
  u\longmapsto(u(q_1),\ldots,u(q_N))
\end{equation}
is an isomorphism.
\end{lemma}

For a K\"ahler metric \(\Omega\) and real-valued functions \(\psi,\phi\), define
\begin{equation}\label{k2:eq:D-operator}
  \mathscr D_{\Omega,\psi}(\phi)
  :=
  g_\Omega^{i\bar j}
  \left(
    \nabla_i\psi\,\nabla_{\bar j}\phi
    +
    \nabla_i\phi\,\nabla_{\bar j}\psi
  \right)
  +
  \phi\Delta_\Omega\psi.
\end{equation}
Write the second-order part of \(L_\Omega\) as
\[
  A_\Omega\phi
  =
  -\Ric(\Omega)^{i\bar j}
   \nabla_i\nabla_{\bar j}\phi.
\]
Its formal adjoint is
\[
  A_\Omega^*\phi
  =
  -\nabla_{\bar j}\nabla_i
  \bigl(\Ric(\Omega)^{i\bar j}\phi\bigr).
\]
Using the K\"ahler contracted Bianchi identity, we obtain
\[
  A_\Omega^*\phi
  =
  A_\Omega\phi
  -
  g_\Omega^{i\bar j}
  \left(
    \nabla_iS(\Omega)\nabla_{\bar j}\phi
    +
    \nabla_i\phi\nabla_{\bar j}S(\Omega)
  \right)
  -
  \phi\Delta_\Omega S(\Omega).
\]
Since the bi-Laplacian is self-adjoint,
\begin{equation}\label{k2:eq:adjoint-difference}
  L_\Omega-L_\Omega^*
  =
  \mathscr D_{\Omega,S(\Omega)}.
\end{equation}
Thus $\mathscr D_{\Omega,S(\Omega)}$ is exactly the difference between $L_\Omega$ and \(L_\Omega^*\).

Recall that $\psi_1,\ldots,\psi_N$ is an $L^2(\omega)$-orthonormal basis of $\mathcal K^*$. Define
\begin{equation}\label{eq:case-operator}
 \mathcal P_\eps=
 \begin{cases}
  L_{\omega_\eps},& k>2,\\[3pt]
  L_{\omega_\eps}^*
  +\mathscr D_{\omega_\eps,
     \pi^*S(\omega)+\eps^2h_S},& k=2,
 \end{cases}
 \qquad
 \widetilde{\mathcal P}_\eps\phi
 =\mathcal P_\eps\phi-
  \sum_{j=1}^N\phi(q_j)\psi_j,
\end{equation}
Through the biholomorphism \(\pi:\widetilde X\setminus E\rightarrow X\setminus S\), we regard both $\mathcal P_\eps$ and $L_\omega$ as operators on $X\setminus S$.  For every compact set $K\Subset X\setminus S$ and every $m\ge0$,
\[
 \norm{(\mathcal P_\eps-L_\omega)\phi}_{C^{m,\alpha}(K)}
 \le c_{K,m}(\eps)
 \norm{\phi}_{C^{m+4,\alpha}(K)},
 \qquad
 c_{K,m}(\eps)\longrightarrow0.
\]
Indeed, when $k>2$, one has $\omega_\eps=\omega$ on $K$ for all sufficiently small $\eps$. When $k=2$,
\[
 \omega_\eps=\omega+\eps^2\ddbar\Gamma
 \longrightarrow\omega
 \qquad\text{in }C^\infty(K),
\]
and hence
\begin{equation}\label{eq:P-L*}
 \mathcal P_\eps
 \longrightarrow
 L_\omega^*+\mathscr D_{\omega,S(\omega)}
 =L_\omega.
\end{equation}
On a chart of scale $\rho\ge\eps$, let $\widehat\omega=\rho^{-2}\omega_\eps$.  Then
\begin{equation}\label{k2:eq:lower-order-scaling}
 \rho^4\mathscr D_{\omega_\eps,
   \pi^*S(\omega)+\eps^2h_S}(\phi)
 =
 O\bigl(
   \rho^3|\nabla_{\widehat\omega}\phi|
   +\rho^4|\phi|
 \bigr).
\end{equation}
Consequently, this additional first-order term converges to zero on every rescaled chart for which $\rho\to0$.

\begin{proposition} \label{prop:linear-inverse}
For all sufficiently small $\eps$, the operator
\[
  \widetilde{\mathcal P}_\eps:
  C^{4,\alpha}_\delta(\widetilde X)\longrightarrow
  C^{0,\alpha}_{\delta-4}(\widetilde X)
\]
is an isomorphism in the following cases:
\begin{enumerate}[label=\textup{(\roman*)}]
\item if $k>2$ and $4-2k<\delta<0$, then
\begin{equation}\label{gt:eq:inverse-bound}
  \norm{\phi}_{C^{4,\alpha}_\delta}
  \le C\norm{\widetilde{\mathcal P}_\eps\phi}_{C^{0,\alpha}_{\delta-4}};
\end{equation}
\item if $k=2$ and $-1<\delta<0$, then
\begin{equation}\label{k2:eq:linear-inverse}
  \norm{\phi}_{C^{4,\alpha}_\delta}
  \le C\eps^\delta
  \norm{\widetilde{\mathcal P}_\eps\phi}_{C^{0,\alpha}_{\delta-4}}.
\end{equation}
\end{enumerate}
\end{proposition}

To prove this estimate, we need several auxiliary lemmas. The first is the weighted Schauder estimate.
\begin{lemma} \label{lem:Schauder}
There is a constant $C$, independent of $\eps$, such that
\begin{equation}
\label{eq:Schauder}
  \norm{\phi}_{C^{4,\alpha}_\delta}
  \le C\left(
    \norm{\phi}_{C^0_\delta}
    +\norm{\widetilde{\mathcal P}_\eps\phi}_{C^{0,\alpha}_{\delta-4}}
  \right).
\end{equation}
\end{lemma}

\begin{proof}
Let \(x\in\widetilde X\) and let \(\rho=\tau(x)\).  On concentric balls of radii comparable to \(\rho\), rescale the metric by \(\rho^{-2}\) and the operator by \(\rho^4\).  The regional estimates in Propositions~\ref{gt:prop:approximate-metric} and \ref{k2:prop:approximate-metric} imply that the rescaled operators are uniformly strongly elliptic with uniformly bounded \(C^{0,\alpha}\)-coefficients.  In codimension two, the same statement for the additional first-order term follows from \eqref{k2:eq:lower-order-scaling}.  The interior fourth-order Schauder estimate, together with \(\tau\simeq\rho\) on these balls, therefore gives
\[
  \norm{\phi}_{C^{4,\alpha}_\delta}
  \le C\left(
    \norm{\phi}_{C^0_\delta}
    +\norm{\mathcal P_\eps\phi}_{C^{0,\alpha}_{\delta-4}}
  \right).
\]
This is also the estimate used in the proof of \cite[Proposition~9]{SS20} when \(k>2\). Finally, since the \(q_j\) lie a fixed positive distance from \(S\) and the \(\psi_j\) are smooth,
\[
 \norm{\sum_{j=1}^N\phi(q_j)\psi_j}_{C^{0,\alpha}_{\delta-4}}
 \le C\norm{\phi}_{C^0_\delta}.
\]
Using \eqref{eq:case-operator} proves \eqref{eq:Schauder}.
\end{proof}

The following lemma is exactly \cite[Lemma~11]{SS20}, which applies for all \(k\ge2\).
\begin{lemma}
\label{lem:Burns-product-kernel}
Let $k\ge2$, let $Y=\Bl_0\C^k\times\C^{n-k}$ carry the scalar-flat product metric $\omega_k$, and let $\delta<0$.  If $v\in C^{4,\alpha}_\delta(Y)$ and \(L_{\omega_k}v=0\), then $v=0$.
\end{lemma}

For \(k>2\), the possible limit on \(X\setminus S\) is eliminated by the following removability argument.
\begin{lemma}\label{gt:lem:compact-model}
Let \(k>2\). Suppose $u\in C^{4,\alpha}_\delta(X\setminus S)$, $\delta>4-2k$, and
\begin{equation}\label{gt:eq:compact-model-equation}
  L_\omega u-\sum_{j=1}^Nu(q_j)\psi_j=0
  \quad\text{on }X\setminus S.
\end{equation}
Then $u=0$.
\end{lemma}

\begin{proof}
As in \cite[Lemma~10]{SS20}, the restriction \(\delta>4-2k\) ensures that \eqref{gt:eq:compact-model-equation} holds distributionally on all of \(X\). Hence \(u\) extends smoothly across \(S\). Pairing \(L_\omega u\) with \(\psi_\ell\in\ker L_\omega^*\) gives
\[
  u(q_\ell)
  =\sum_{j=1}^N u(q_j)
    \langle\psi_j,\psi_\ell\rangle_{L^2}
  =\langle L_\omega u,\psi_\ell\rangle_{L^2}
  =0.
\]
Thus \(L_\omega u=0\) and \(\ev_q(u)=0\).  Since \(\ev_q\) is injective, \(u=0\).
\end{proof}

The following weighted Liouville theorem on the punctured Euclidean product controls the corresponding neck rescaling, which is \cite[Lemma~12]{SS20}.

\begin{lemma}\label{gt:lem:euclidean-model}
Let $Z=(\C^k\setminus\{0\})\times\C^{n-k}$ carry the Euclidean product metric. If $v\in C^{4,\alpha}_\delta(Z)$, $4-2k<\delta<0$, and \(\Delta_0^2v=0\), then $v=0$.
\end{lemma}

In codimension two, passing to distributional limits requires the following volume and test-function estimates.
\begin{lemma}\label{k2:lem:volume-test}
For $\eps\le r\le r_0$, we have
\begin{equation}\label{k2:eq:volume-estimates}
  \Vol_{\omega_\eps}\{r<d<2r\}\le Cr^4,
  \qquad
  \Vol_{\omega_\eps}\{d<2\eps\}\le C\eps^4.
\end{equation}
Moreover, for every fixed $\chi\in C^\infty(X)$, pulled back to $\widetilde X$, we have
\begin{equation}\label{k2:eq:test-function-estimate}
  |\mathcal P_\eps^*\chi|\le C_\chi\tau^{-3}.
\end{equation}
\end{lemma}
\begin{proof}
By Proposition \ref{k2:prop:approximate-metric}, \(\omega_\eps\) is a small perturbation of the scalar-flat model product metric near \(S\). Hence there are constants \(c,C>0\), independent of \(\eps\), such that, for every \(x\in\widetilde X\), with \(\rho=\tau_\eps(x)\),
\begin{equation}\label{k2:eq:local-volume-bound}
  \Vol_{\omega_\eps}
  \bigl(B_{\omega_\eps}(x,c\rho)\bigr)
  \le C\rho^{2n}.
\end{equation}
Since \(S\) has real dimension \(2n-4\), we can cover \(S\) by at most \(Cr^{-(2n-4)}\) balls of radius \(r\).  Over each such ball, the four-dimensional normal annulus \(\{r<d<2r\}\) is covered by a uniformly bounded number of balls of radius comparable to \(r\).  Therefore
\[
  \Vol_{\omega_\eps}\{r<d<2r\}
  \le Cr^{-(2n-4)}r^{2n}
  =Cr^4.
\]
The same argument at scale \(r=\eps\), using the rescaled blowup charts \eqref{gt:eq:blowup-coordinates}, gives
\[
  \Vol_{\omega_\eps}\{d<2\eps\}\le C\eps^4.
\]

We next prove \eqref{k2:eq:test-function-estimate}. Write \(h_\eps=\pi^*S(\omega)+\eps^2h_S\). Fix \(x\in\widetilde X\), let \(\rho=\tau(x)\), and rescale \(\omega_\eps\) to \(\widehat\omega=\rho^{-2}\omega_\eps\).  Since \(\chi\) and \(h_\eps\) are pullbacks of uniformly smooth functions on \(X\), we have
\[
  \norm{\chi-\chi(\pi(x))}
       _{C^{4,\alpha}(\widehat\omega)}
  \le C_\chi\rho,
  \qquad
  |\nabla_{\widehat\omega}h_\eps|\le C\rho,
  \qquad
  |\nabla_{\widehat\omega}^2h_\eps|\le C\rho^2.
\]
The rescaled operator \(\rho^4L_{\omega_\eps}\) has uniformly bounded coefficients and annihilates constants.  Hence
\[
  |\rho^4L_{\omega_\eps}\chi|\le C_\chi\rho.
\]
Moreover, \(\mathscr D_{\omega_\eps,h_\eps}^*\chi\) contains only terms of the form \(\langle\nabla h_\eps,\nabla\chi\rangle_{\omega_\eps}\) and \(\chi\,\Delta_{\omega_\eps}h_\eps\). Therefore
\[
  |\rho^4\mathscr D_{\omega_\eps,h_\eps}^*\chi|
  \le C_\chi\rho^4.
\]
Consequently, by \eqref{eq:case-operator},
\[
  |\rho^4\mathcal P_\eps^*\chi|
  \le C_\chi\rho,
\]
and hence
\[
  |\mathcal P_\eps^*\chi|
  \le C_\chi\rho^{-3}
  =C_\chi\tau(x)^{-3}.
\]
\end{proof}

\begin{lemma}\label{k2:lem:compact-limit}
Let $\eps_i\to0$ and suppose
\begin{equation}\label{k2:eq:compact-limit-assumption}
  \norm{\phi_i}_{C^0_\delta}=1,
  \qquad
  \eps_i^\delta
  \norm{\widetilde{\mathcal P}_{\eps_i}\phi_i}_{C^{0,\alpha}_{\delta-4}}
  \longrightarrow0.
\end{equation}
If $\phi_i\to u$ locally on $X\setminus S$, then $u=0$.
\end{lemma}

\begin{proof}
Let \(\tau_i=\tau_{\eps_i}\), \(F_i=\widetilde{\mathcal P}_{\eps_i}\phi_i\). Choose \(M_i\in\mathbb N\) such that
\(2^{M_i}\eps_i\le r_0<2^{M_i+1}\eps_i\). Using \eqref{k2:eq:volume-estimates}, we have
\begin{align}
 \int_{\widetilde X}\tau_i^{\delta-4}\omega_{\eps_i}^n
 &=
 \int_{\{d<2\eps_i\}}\tau_i^{\delta-4}\omega_{\eps_i}^n
 +\sum_{m=1}^{M_i}
   \int_{\{2^m\eps_i\le d<2^{m+1}\eps_i\}}
     \tau_i^{\delta-4}\omega_{\eps_i}^n 
 +\int_{\{d\ge2^{M_i+1}\eps_i\}}
     \tau_i^{\delta-4}\omega_{\eps_i}^n \notag\\
 &\le
 C\eps_i^\delta
 +C\sum_{m=1}^{M_i}(2^m\eps_i)^\delta+C
 \le C\eps_i^\delta.
 \label{k2:eq:target-integral}
\end{align}
Here the last inequality uses \(\delta<0\) and
\(\eps_i^\delta\ge1\). Consequently, for every \(\chi\in C^\infty(X)\),
\[
  \left|
    \int_{\widetilde X}F_i\chi\,\omega_{\eps_i}^n
  \right|
  \le C_\chi\eps_i^\delta
       \norm{F_i}_{C^0_{\delta-4}}
  \longrightarrow0.
\]
Since \(|\phi_i|\le\tau_{\eps_i}^\delta\), Lemma~\ref{k2:lem:volume-test} gives
\[
  |\phi_i\mathcal P_{\eps_i}^*\chi|
  \le C_\chi\tau_{\eps_i}^{\delta-3}.
\]
For \(\eps_i<\rho<r_0\), using \(\delta>-1\), the same decomposition gives
\begin{equation}\label{k2:eq:adjoint-tail}
 \int_{\{d<\rho\}}\tau_i^{\delta-3}\omega_{\eps_i}^n
 \le
 C\eps_i^{\delta+1}
 +C\sum_{\substack{m\ge1\\2^m\eps_i<\rho}}
   (2^m\eps_i)^{\delta+1}
 \le C\rho^{\delta+1}.
\end{equation}
By definition of \(\widetilde{\mathcal P}_{\eps_i}\), we have
\[
  \int_{\widetilde X}F_i\chi\,\omega_{\eps_i}^n
  =
  \int_{\widetilde X}
    \phi_i\mathcal P_{\eps_i}^*\chi\,\omega_{\eps_i}^n
  -
  \sum_{j=1}^N\phi_i(q_j)
    \int_{\widetilde X}\psi_j\chi\,\omega_{\eps_i}^n.
\]
On every compact subset of \(X\setminus S\), by Proposition \ref{k2:prop:approximate-metric}, \eqref{eq:P-L*} together with the assumed local convergence, we have
\[
  \phi_i\longrightarrow u,
  \qquad
  \mathcal P_{\eps_i}^*\chi\longrightarrow L_\omega^*\chi,
  \qquad
  \omega_{\eps_i}^n\longrightarrow\omega^n.
\]
Using \eqref{k2:eq:adjoint-tail}, we may first pass to the limit away from \(S\) and then let \(\rho\searrow0\) to obtain
\[
  \int_XuL_\omega^*\chi\,\omega^n
  =
  \sum_{j=1}^Nu(q_j)\int_X\psi_j\chi\,\omega^n.
\]
Thus
\[
  L_\omega u-\sum_{j=1}^Nu(q_j)\psi_j=0
\]
distributionally on \(X\). Since \(|u|\le d^\delta\), \(u\) is locally integrable, and elliptic regularity shows that \(u\) is smooth.  Pairing \(L_\omega u\) with \(\psi_\ell\in\ker L_\omega^*\) gives \(u(q_\ell)=0\). Hence \(u\in\mathcal K\) and \(\ev_q(u)=0\). Consequently, \(u=0\).
\end{proof}

In codimension two, an analogue of Lemma \ref{gt:lem:euclidean-model} is false, so we need the following.
\begin{lemma}\label{k2:lem:matched-neck}
Assume \eqref{k2:eq:compact-limit-assumption}. Let \(x_i\in\widetilde X\) satisfy \(\tau_{\eps_i}(x_i)\rightarrow0\) and \(\eps_i/\tau_{\eps_i}(x_i)\rightarrow0\). Choose coordinates centered at the nearest point of \(x_i\) on \(S\) and rescale by \(\tau_{\eps_i}(x_i)\). If \(\tau_{\eps_i}(x_i)^{-\delta}\phi_i\) converges locally on \((\C^2\setminus\{0\})\times\C^{n-2}\) to a function \(v\), then \(v\) extends smoothly across \(\{0\}\times\C^{n-2}\) to an biharmonic function on \(\C^n\). Moreover, \(v=0\).
\end{lemma}
\begin{proof}
Let \(\rho_i=\tau_{\eps_i}(x_i)\), \(a_i=\eps_i/\rho_i\) and \(F_i=\widetilde{\mathcal P}_{\eps_i}\phi_i\). Away from \(\{z=0\}\), the rescaled metrics \(\rho_i^{-2}\omega_{\eps_i}\) converge locally to the Euclidean metric, while \(\rho_i^4\mathcal P_{\eps_i}\) converges locally to \(-\frac14\Delta_0^2\). Therefore,
\[
  \Delta_0^2v=0
  \qquad\text{on }
  (\C^2\setminus\{0\})\times\C^{n-2}.
\]
It remains to show that no distributional term is supported on \(\{z=0\}\).

Let \(\chi\in C_c^\infty(\C^n)\). For sufficiently large \(i\), its pullback is a smooth test function on the rescaled blowup. Formal adjointness on the whole smooth manifold gives
\begin{align}
  \int_{\widetilde X}
    \rho_i^{-\delta}\phi_i\,
    \rho_i^4\mathcal P_{\eps_i}^*\chi\,
    (\rho_i^{-2}\omega_{\eps_i})^n=&
    \int_{\widetilde X}
    \rho_i^{4-\delta}F_i\chi\,
    (\rho_i^{-2}\omega_{\eps_i})^n \notag\\&+
    \int_{\widetilde X}
    \rho_i^{4-\delta}
    \sum_{j=1}^N\phi_i(q_j)\psi_j\chi\,
    (\rho_i^{-2}\omega_{\eps_i})^n.
  \label{k2:eq:matched-adjoint}
\end{align}
In the rescaled coordinates, \(\tau_{\eps_i}/\rho_i\simeq \max\{a_i,|z|\}\). Therefore,
\[
  \bigl|\rho_i^{4-\delta}F_i\bigr|
  \le
  C\norm{F_i}_{C^0_{\delta-4}}\max\{a_i,|z|\}^{\delta-4}.
\]
Choose \(R>1\) such that \(\supp\chi\subset\{|z|<R,\ |w|<R\}\). By rescaling \eqref{k2:eq:local-volume-bound}, a ball of radius comparable to \(s\) in the rescaled metric has volume at most \(Cs^{2n}\). The support in the tangential direction \(\{|w|<R\}\) can be covered by \(C_\chi s^{-(2n-4)}\) such balls. Hence, for \(a_i\le s\le R\), we have
\[
  \int_{\supp\chi\cap\{s\le|z|<2s\}}
    (\rho_i^{-2}\omega_{\eps_i})^n
  \le C_\chi s^4.
\]
The same argument at the cap scale gives
\[
  \int_{\supp\chi\cap\{|z|<2a_i\}}
    (\rho_i^{-2}\omega_{\eps_i})^n
  \le C_\chi a_i^4.
\]
Decomposing \(\supp\chi\) into this cap and the normal annuli \(\{2^ma_i\le|z|<2^{m+1}a_i\}\), we obtain
\begin{align*}
  &\int_{\supp\chi}
    \max\{a_i,|z|\}^{\delta-4}
    (\rho_i^{-2}\omega_{\eps_i})^n\\
  &\quad=
    \int_{\supp\chi\cap\{|z|<2a_i\}}
      \max\{a_i,|z|\}^{\delta-4}
      (\rho_i^{-2}\omega_{\eps_i})^n\\
  &\qquad+
    \sum_{m=1}^{\infty}
    \int_{\supp\chi\cap
      \{2^ma_i\le |z|<2^{m+1}a_i\}}
      \max\{a_i,|z|\}^{\delta-4}
      (\rho_i^{-2}\omega_{\eps_i})^n\\
  &\quad\le
    C_\chi a_i^{\delta-4}a_i^4
    +C_\chi\sum_{m=1}^{\infty}
      (2^ma_i)^{\delta-4}(2^ma_i)^4\\
  &\quad=
    C_\chi a_i^\delta
    \left(1+\sum_{m=1}^{\infty}2^{m\delta}\right)
    \le C_\chi a_i^\delta.
\end{align*}
Consequently, by \eqref{k2:eq:compact-limit-assumption}, we have
\begin{align}
  \left|
    \int_{\widetilde X}
    \rho_i^{4-\delta}F_i\chi\,
    (\rho_i^{-2}\omega_{\eps_i})^n
  \right|
  \le C\norm{F_i}_{C^0_{\delta-4}} a_i^\delta 
  =C\rho_i^{-\delta}\eps_i^\delta \norm{F_i}_{C^0_{\delta-4}}
  \longrightarrow0.
  \label{k2:eq:matched-source}
\end{align}
Since \(\phi_i(q_j)\) is uniformly bounded and \(\psi_j\) is smooth, we have
\begin{equation}\label{k2:eq:matched-finite-rank}
  \left|
    \int_{\widetilde X}
      \rho_i^{4-\delta}
      \sum_{j=1}^N\phi_i(q_j)\psi_j\chi\,
      (\rho_i^{-2}\omega_{\eps_i})^n
  \right|
  \le C_\chi\rho_i^{4-\delta}
  \longrightarrow0.
\end{equation}

We next control the part of the left-hand side of \eqref{k2:eq:matched-adjoint} that collapses to \(\{z=0\}\). Since \(\norm{\phi_i}_{C^0_\delta}=1\), we have
\[
  \bigl|\rho_i^{-\delta}\phi_i\bigr|
  \le C\max\{a_i,|z|\}^{\delta}.
\]
The argument in proving \eqref{k2:eq:test-function-estimate}, applied in the rescaled coordinates, gives
\[
  \bigl|\rho_i^4\mathcal P_{\eps_i}^*\chi\bigr|
  \le C_\chi\max\{a_i,|z|\}^{-3}.
\]
Fix \(0<\eta<1\), independent of \(i\). Since \(a_i\to0\), we have \(a_i<\eta\) for all sufficiently large \(i\). Therefore,
\begin{equation}\label{k2:eq:matched-tail}
  \int_{\{|z|<\eta\}}
  \left|
    \rho_i^{-\delta}\phi_i\,
    \rho_i^4\mathcal P_{\eps_i}^*\chi
  \right|
  (\rho_i^{-2}\omega_{\eps_i})^n
  \le
  C_\chi\bigl(\eta^{\delta+1}+a_i^{\delta+1}\bigr).
\end{equation}
Here the first term comes from the normal annuli \(a_i<|z|<\eta\), while the second comes from the cap \(|z|\lesssim a_i\). Split the left-hand side of \eqref{k2:eq:matched-adjoint} into the regions \(\{|z|<\eta\}\) and \(\{|z|\ge\eta\}\). On \(\{|z|\ge\eta\}\), it is clear that
\[
  \int_{\{|z|\ge\eta\}}
    \rho_i^{-\delta}\phi_i\,
    \rho_i^4\mathcal P_{\eps_i}^*\chi\,
    (\rho_i^{-2}\omega_{\eps_i})^n
  \longrightarrow
  \int_{\{|z|\ge\eta\}}
    v\,\Delta_0^2\chi\,\omega_0^n.
\]
Combining \eqref{k2:eq:matched-adjoint}, \eqref{k2:eq:matched-source}, \eqref{k2:eq:matched-finite-rank}, and
\eqref{k2:eq:matched-tail}, we obtain
\begin{equation}\label{k2:eq:matched-exterior-limit}
  \left|
    \int_{\{|z|\ge\eta\}}
      v\,\Delta_0^2\chi\,\omega_0^n
  \right|
  \le C_\chi\eta^{\delta+1}.
\end{equation}

Since \(\norm{\phi_i}_{C^0_\delta}=1\), passing to the limit gives
\begin{equation}\label{k2:eq:matched-limit-growth}
  |v(z,w)|\le C|z|^\delta
  \qquad\text{for }z\ne0.
\end{equation}
Since the normal variable \(z\in\C^2\) and \(\delta>-1\), \eqref{k2:eq:matched-limit-growth} implies \(v\in L^1_{\mathrm{loc}}(\C^n)\). In particular,
\begin{equation}\label{k2:eq:matched-interior-limit}
  \int_{\{|z|<\eta\}}
    |v\,\Delta_0^2\chi|\,\omega_0^n
  \le C_\chi\eta^{\delta+4}
  \longrightarrow0.
\end{equation}
Combining \eqref{k2:eq:matched-exterior-limit} and \eqref{k2:eq:matched-interior-limit}, and letting \(\eta\searrow0\), we obtain
\begin{equation}\label{k2:eq:matched-distributional-equation}
  \int_{\C^n}v\,\Delta_0^2\chi\,\omega_0^n=0
  \qquad
  \text{for every }\chi\in C_c^\infty(\C^n).
\end{equation}
Thus \(v\) is distributionally biharmonic on \(\C^n\). Elliptic regularity applied to \eqref{k2:eq:matched-distributional-equation} shows that \(v\) extends smoothly across \(\{z=0\}\) as a biharmonic function.

It remains to prove that this extension vanishes. For every \(w_0\in\C^{n-2}\), \eqref{k2:eq:matched-limit-growth} gives
\[
  \int_{B_2((0,w_0))}
    |v|\,\omega_0^n
  \le C\int_{\{|z|<2\}}|z|^\delta\,\omega_0^2
  \le C,
\]
where the constant is independent of \(w_0\). The interior estimate for biharmonic functions gives
\[
  \sup_{B_1((0,w_0))}|v|
  \le
  C\int_{B_2((0,w_0))}|v|\,\omega_0^n
  \le C.
\]
Thus \(v\) is uniformly bounded near \(\{0\}\times\C^{n-2}\). On the region \(|z|\ge1\), \eqref{k2:eq:matched-limit-growth} also gives \(|v|\le C\) since \(\delta \in (-1,0)\). Hence \(v\) is a bounded biharmonic function. In particular, \(v\) defines a tempered distribution. Taking the Fourier transform of \eqref{k2:eq:matched-distributional-equation} gives \(|\xi|^4\mathcal Fv=0\). It follows that \(\mathcal Fv\) is supported at the origin, and hence that \(v\) is a polynomial. Since \(v\) is bounded, it is constant. Finally,
letting \(|z|\to\infty\) in \eqref{k2:eq:matched-limit-growth} and using \(\delta<0\) shows that this constant is zero. Therefore \(v=0\).
\end{proof}

We now combine these lemmas to prove the uniform linear estimate.
\begin{proof}[Proof of Proposition~\ref{prop:linear-inverse}]
Assume first that \(k>2\). The argument in \cite[Proposition~9]{SS20} applies without substantive change. Lemma~\ref{lem:Schauder} supplies the local estimate, while Lemmas~\ref{gt:lem:compact-model}, \ref{lem:Burns-product-kernel}, and \ref{gt:lem:euclidean-model} exclude the compact, bubble, and neck limits, respectively. The only difference is the finite-rank term in \eqref{eq:case-operator}, which is incorporated into Lemma~\ref{gt:lem:compact-model} and vanishes under the bubble and neck rescalings. This proves \eqref{gt:eq:inverse-bound}.

Assume now that \(k=2\). By Lemma~\ref{lem:Schauder}, it suffices to prove
\begin{equation}\label{k2:eq:C0-linear}
  \norm{\phi}_{C^0_\delta}
  \le
  C\eps^\delta
  \norm{\widetilde{\mathcal P}_\eps\phi}
       _{C^{0,\alpha}_{\delta-4}}.
\end{equation}
Suppose that \eqref{k2:eq:C0-linear} fails. After normalization, there exist \(\eps_i\to0\) and \(\phi_i\) satisfying \eqref{k2:eq:compact-limit-assumption}. Let \(\tau_i=\tau_{\eps_i}\) and \(F_i=\widetilde{\mathcal P}_{\eps_i}\phi_i\). Since \(\delta<0\), \eqref{k2:eq:compact-limit-assumption} gives
\begin{equation}\label{eq:F-to-0}
  \norm{F_i}_{C^{0,\alpha}_{\delta-4}}
  =
  \eps_i^{-\delta}
  \left(
    \eps_i^\delta
    \norm{F_i}_{C^{0,\alpha}_{\delta-4}}
  \right)
  \longrightarrow0.
\end{equation}
Lemma~\ref{lem:Schauder} therefore gives 
\begin{equation}\label{k2:eq:uniform-holder-bound}
  \norm{\phi_i}_{C^{4,\alpha}_\delta}\le C.
\end{equation}
For every compact subset \(K\Subset X\setminus S\), the weight \(\tau_i\) is bounded below by a positive constant depending only on \(K\), and \(\omega_{\eps_i}\) converges smoothly to \(\omega\) on \(K\). Hence \eqref{k2:eq:uniform-holder-bound} gives
\[
  \norm{\phi_i}_{C^{4,\alpha}(K,\omega)}\le C_K.
\]
By the Arzelà--Ascoli theorem and the compact embedding \(C^{4,\alpha}\hookrightarrow C^{4,\alpha'}\), after passing to a subsequence, \(\phi_i\) converges locally on \(X\setminus S\) to a function \(u\) in \(C^{4,\alpha'}\). Lemma~\ref{k2:lem:compact-limit} gives \(u=0\). In particular,
\begin{equation}\label{k2:eq:compact-q-values}
  \phi_i(q_j)\longrightarrow0
  \qquad (1\le j\le N).
\end{equation}
Choose \(x_i\in\widetilde X\) such that
\begin{equation}\label{k2:eq:normalizing-points}
  \tau_i(x_i)^{-\delta}|\phi_i(x_i)|
  \ge\frac12,
\end{equation}
and write \(\rho_i=\tau_i(x_i)\). The local convergence \(\phi_i\to0\) on \(X\setminus S\), together with
\eqref{k2:eq:normalizing-points}, implies that \(\rho_i\to0\). After passing to a further subsequence, either
\(\rho_i/\eps_i\) is bounded or \(\rho_i/\eps_i\to\infty\).
~\\

\emph{Case 1: \(\rho_i/\eps_i\) is bounded.}
Choose \(C>1\) such that \(1\le\rho_i/\eps_i\le C\). In coordinates centered at the nearest point of \(S\), dilate by \(\eps_i\). The metrics \(\eps_i^{-2}\omega_{\eps_i}\) converge locally smoothly to the product model metric \(\omega_2\) on \(Y=\Bl_0\C^2\times\C^{n-2}\). By \eqref{k2:eq:uniform-holder-bound}, and passing to a subsequence, we have
\[
  \eps_i^{-\delta}\phi_i
  \longrightarrow v_1
  \quad\text{in }C^{4,\alpha'}_{\mathrm{loc}}(Y)
\]
for every \(\alpha'<\alpha\). In particular,  \(v_1\in C^{4,\alpha'}_\delta(Y)\). Note that \(v_1\) is nonzero. Indeed, the rescaled points \(x_i\) remain in a fixed compact subset of \(Y\), and
\begin{align}
  \eps_i^{-\delta}|\phi_i(x_i)|
  =
  \left(\frac{\rho_i}{\eps_i}\right)^\delta
  \rho_i^{-\delta}|\phi_i(x_i)| 
  \ge \frac12 C^\delta>0
  \label{k2:eq:Burns-limit-nonzero}
\end{align}
by \eqref{k2:eq:normalizing-points} and \(\delta<0\).

In the \(\eps_i\)-rescaled coordinates, the equation
\[
  \mathcal P_{\eps_i}\phi_i
  =
  F_i+\sum_{j=1}^N\phi_i(q_j)\psi_j
\]
becomes
\begin{equation}\label{k2:eq:Burns-rescaled-equation}
  \bigl(\eps_i^4\mathcal P_{\eps_i}\bigr)
  \bigl(\eps_i^{-\delta}\phi_i\bigr)
  =
  \eps_i^{4-\delta}F_i
  +
  \eps_i^{4-\delta}
  \sum_{j=1}^N\phi_i(q_j)\psi_j.
\end{equation}
For every compact subset \(K\subset Y\), the definition of the weighted \(C^{0,\alpha}_{\delta-4}\)-norm and \eqref{eq:F-to-0} give
\begin{equation}\label{k2:eq:Burns-source-vanishing}
  \norm{\eps_i^{4-\delta}F_i}
       _{C^{0,\alpha}(K,\eps_i^{-2}\omega_{\eps_i})}
  \le
  C_K\norm{F_i}_{C^{0,\alpha}_{\delta-4}}
  \longrightarrow0.
\end{equation}
Moreover, \eqref{k2:eq:compact-q-values} and the smoothness of the \(\psi_j\) give
\begin{equation}\label{k2:eq:Burns-finite-rank-vanishing}
  \norm{
    \eps_i^{4-\delta}
    \sum_{j=1}^N\phi_i(q_j)\psi_j
  }_{C^{0,\alpha}(K,\eps_i^{-2}\omega_{\eps_i})}
  \longrightarrow0.
\end{equation}
By Proposition~\ref{k2:prop:approximate-metric}, the smooth convergence \(\eps_i^{-2}\omega_{\eps_i}\rightarrow\omega_2\) on compact subsets of \(Y\) implies convergence of the corresponding fourth-order operators
\[
  \eps_i^4L_{\omega_{\eps_i}}^*
  \longrightarrow L_{\omega_2}^*.
\]
Since the product model metric is scalar-flat, equation \eqref{k2:eq:adjoint-difference} gives \(L_{\omega_2}^*=L_{\omega_2}\). Moreover, by \eqref{eq:case-operator} and
\eqref{k2:eq:lower-order-scaling},
\[
  \eps_i^4\mathcal P_{\eps_i}
  =
  \eps_i^4L_{\omega_{\eps_i}}^*
  +
  \eps_i^4
  \mathscr D_{\omega_{\eps_i},
    \pi^*S(\omega)+\eps_i^2h_S}
  \longrightarrow L_{\omega_2}
\]
on compact subsets of \(Y\). Passing to the limit in \eqref{k2:eq:Burns-rescaled-equation}, using \eqref{k2:eq:Burns-source-vanishing} and \eqref{k2:eq:Burns-finite-rank-vanishing}, gives
\(L_{\omega_2}v_1=0\). By Lemma~\ref{lem:Burns-product-kernel}, we have \(v_1=0\), contradicting \eqref{k2:eq:Burns-limit-nonzero}.
~\\

\emph{Case 2: \(\rho_i/\eps_i\to\infty\).}
Clearly, \(\eps_i/\rho_i \to 0\). In coordinates centered at the nearest point of \(S\), dilate by \(\rho_i\). The metrics \(\rho_i^{-2}\omega_{\eps_i}\) converge locally smoothly to the Euclidean product metric \(\omega_0\) on \(Z=(\C^2\setminus\{0\})\times\C^{n-2}\). By \eqref{k2:eq:uniform-holder-bound}, and passing to a subsequence, we have
\[
  \rho_i^{-\delta}\phi_i
  \longrightarrow v_2
  \quad\text{in }C^{4,\alpha'}_{\mathrm{loc}}(Z)
\]
for every \(\alpha'<\alpha\). In particular, \(v_2\in C^{4,\alpha'}_\delta(Z)\). Note that \(v_2\) is nonzero. Indeed, the rescaled points \(x_i\) remain in a fixed normal annulus, and by \eqref{k2:eq:normalizing-points},
\begin{equation}\label{k2:eq:neck-limit-nonzero}
  \rho_i^{-\delta}|\phi_i(x_i)|\ge\frac12.
\end{equation}

As in Case~1, \eqref{eq:F-to-0} and \eqref{k2:eq:compact-q-values} show that the rescaled source and finite-rank terms vanish on compact subsets of \(Z\), while \eqref{k2:eq:lower-order-scaling} removes the additional first-order term. Proposition~\ref{k2:prop:approximate-metric} and \eqref{k2:eq:adjoint-difference} therefore give
\[
  \rho_i^4\mathcal P_{\eps_i}
  \longrightarrow L_{\omega_0}=
-\Delta_0^2
\]
on compact subsets of \(Z\), and passing to the limit in the rescaled equation yields \(L_{\omega_0}v_2=0\).

More importantly, \eqref{k2:eq:compact-limit-assumption}, the limits \(\rho_i\rightarrow0\), \(\eps_i/\rho_i\rightarrow0\), and the local convergence to \(v_2\) verify all the hypotheses of Lemma~\ref{k2:lem:matched-neck}. Hence \(v_2=0\), contradicting \eqref{k2:eq:neck-limit-nonzero}.

Both cases are impossible. Hence \eqref{k2:eq:C0-linear} holds. Combining it with Lemma~\ref{lem:Schauder}, and using \(\eps^\delta\ge1\) for small \(\eps\), gives
\begin{align*}
  \norm{\phi}_{C^{4,\alpha}_\delta}
  \le
  C\left(
    \norm{\phi}_{C^0_\delta}
    +
    \norm{\widetilde{\mathcal P}_\eps\phi}
         _{C^{0,\alpha}_{\delta-4}}
  \right)
  \le
  C\eps^\delta
  \norm{\widetilde{\mathcal P}_\eps\phi}
       _{C^{0,\alpha}_{\delta-4}}.
\end{align*}
This proves \eqref{k2:eq:linear-inverse}.

For each fixed \(\eps\), the weighted H\"older norms are equivalent to the ordinary H\"older norms on the compact manifold \(\widetilde X\). The operator \(\mathcal P_\eps\) is a lower-order perturbation of the index-zero elliptic operator \(L_{\omega_\eps}^*\), and hence has index zero. Since \(\widetilde{\mathcal P}_\eps\) is a finite-rank perturbation of \(\mathcal P_\eps\), it also has index zero. Estimates \eqref{gt:eq:inverse-bound} and \eqref{k2:eq:linear-inverse}, in their respective cases, imply \(\ker\widetilde{\mathcal P}_\eps=\{0\}\). Therefore \(\widetilde{\mathcal P}_\eps\) is surjective and hence an isomorphism.
\end{proof}

\subsection{The nonlinear equation}
\label{sec:nonlinear}

We now perturb the approximation metric \(\omega_\eps\). Let \(\omega_{\eps,\phi}= \omega_\eps+\ddbar\phi\). The scalar curvature linearization at \(\omega_\eps\) is
\begin{equation}\label{eq:scalar-expansion}
  S(\omega_{\eps,\phi})
  =
  S(\omega_\eps)+L_{\omega_\eps}\phi+Q_\eps(\phi).
\end{equation}
Our goal is to solve the modified scalar-curvature equation
\begin{equation}\label{eq:nonlinear-target}
  S(\omega_{\eps,\phi})
  =
  \begin{cases}
    \displaystyle
    \pi^*S(\omega)
    +\sum_{j=1}^N\phi(q_j)\psi_j,
    & k>2,\\[8pt]
    \displaystyle
    \pi^*S(\omega)+\eps^2h_S
    +\sum_{j=1}^N\phi(q_j)\psi_j,
    & k=2.
  \end{cases}
\end{equation}
The finite-rank terms arise from the definition of \(\widetilde{\mathcal P}_\eps\). Once \(\norm{\phi}_{C^0_\delta}\to0\), their coefficients tend to zero because the points \(q_j\) are away from \(S\). In codimension two, the additional term \(\eps^2h_S\) also tends uniformly to zero. Thus a sufficiently small solution of \eqref{eq:nonlinear-target} gives the desired scalar-curvature approximation. We first show the nonlinear estimate used for both cases.

\begin{lemma}\label{lem:quadratic}
There are constants \(c,C>0\), independent of \(\eps\), such that if
\[
  \norm{\phi}_{C^{4,\alpha}_2}\le c,
  \qquad
  \norm{\psi}_{C^{4,\alpha}_2}\le c,
\]
then
\begin{align}
  \norm{Q_\eps(\phi)-Q_\eps(\psi)}_{C^{0,\alpha}_{\delta-4}}
  &\le
  C\left(
    \norm{\phi}_{C^{4,\alpha}_2}
    +\norm{\psi}_{C^{4,\alpha}_2}
  \right)
  \norm{\phi-\psi}_{C^{4,\alpha}_\delta},
  \label{eq:quadratic-difference}\\
  \norm{Q_\eps(\phi)}_{C^{0,\alpha}_{\delta-4}}
  &\le
  C\norm{\phi}_{C^{4,\alpha}_2}
  \norm{\phi}_{C^{4,\alpha}_\delta}.
  \label{eq:quadratic-single}
\end{align}
\end{lemma}

\begin{proof}
Fix a weighted ball centered at \(x\), and let \(\rho=\tau(x)\). Denote by
\[
  \widehat\omega_\eps=\rho^{-2}\omega_\eps,
  \qquad
  \widehat\phi=\rho^{-2}\phi,
  \qquad
  \widehat\psi=\rho^{-2}\psi.
\]
The rescaled metrics have uniformly bounded geometry. Moreover, by the assumptions
\[
  \|\widehat\phi\|_{C^{4,\alpha}}
  +
  \|\widehat\psi\|_{C^{4,\alpha}}
  \le Cc.
\]
Choose \(c>0\) sufficiently small. Then, for every \(t\in[0,1]\), the metrics
\[
  \widehat\omega_\eps+t\ddbar\widehat\phi,
  \qquad
  \widehat\omega_\eps+t\ddbar\widehat\psi,\qquad
  \widehat\omega_\eps
  +\ddbar\bigl((1-t)\widehat\psi+t\widehat\phi\bigr)
\]
are K\"ahler and uniformly equivalent to \(\widehat\omega_\eps\). Consequently, on this fixed \(C^{4,\alpha}\)-neighbourhood of the origin, the second derivative of
\[
  u\longmapsto S(\widehat\omega_\eps+\ddbar u)
\]
is bounded uniformly in \(\eps\) and in the rescaled chart. The mean-value formula therefore gives
\[
  \|
    \widehat Q_\eps(\widehat\phi)
    -\widehat Q_\eps(\widehat\psi)
  \|_{C^{0,\alpha}}
  \le
  C\left(\|\widehat\phi\|_{C^{4,\alpha}}+\|\widehat\psi\|_{C^{4,\alpha}}\right)
  \|\widehat\phi-\widehat\psi\|_{C^{4,\alpha}}.
\]
Moreover,
\[
  \|\widehat\phi\|_{C^{4,\alpha}}
  \le C\norm{\phi}_{C^{4,\alpha}_2},
  \qquad
  \|\widehat\phi-\widehat\psi\|_{C^{4,\alpha}}
  \le
  C\rho^{\delta-2}
  \norm{\phi-\psi}_{C^{4,\alpha}_\delta}.
\]
Since scalar curvature acquires the factor \(\rho^{-2}\) when scaling
back, the resulting bound has size \(\rho^{\delta-4}\). Taking the
supremum over all weighted balls proves
\eqref{eq:quadratic-difference}. Setting \(\psi=0\) gives
\eqref{eq:quadratic-single}.
\end{proof}
Since \(\tau\ge\eps\) and \(\delta-2<0\), the definition of the weighted
norms also gives
\begin{equation}\label{eq:Cdelta-C2}
  \norm{\phi}_{C^{4,\alpha}_2}
  \le
  C\eps^{\delta-2}
  \norm{\phi}_{C^{4,\alpha}_\delta}.
\end{equation}

We next show an additional estimate needed only in the codimension two case. Choose
\begin{equation}\label{k2:eq:nonlinear-delta-choice}
  \frac12<\beta<\frac23,
  \qquad
  \frac{2-4\beta}{3-\beta}<\delta<0.
\end{equation}
The second inequality is equivalent to
\begin{equation}\label{k2:eq:nonlinear-exponent}
  3\delta-2+\beta(4-\delta)>0.
\end{equation}
The interval in \eqref{k2:eq:nonlinear-delta-choice} is contained in \((-1,0)\), because
\[
  \frac{2-4\beta}{3-\beta}>-1
  \quad\Longleftrightarrow\quad
  \beta<1,
\]
and it is nonempty because \(\beta>1/2\).

\begin{lemma}
\label{k2:lem:remainder-estimate}
Assume \(k=2\) and \(-1<\delta<0\). Then, for every \(\phi\in C^{4,\alpha}_\delta(\widetilde X)\),
\begin{equation}\label{k2:eq:remainder-estimate}
  \norm{
    \mathscr D_{\omega_\eps,
      S(\omega_\eps)-\pi^*S(\omega)-\eps^2h_S}(\phi)
  }_{C^{0,\alpha}_{\delta-4}}
  \le
  C\eps^{\delta-2}r_\eps^{4-\delta}
  \norm{\phi}_{C^{4,\alpha}_\delta}.
\end{equation}
\end{lemma}

\begin{proof}
Throughout this proof, \(k=2\). Let \(e_\eps=S(\omega_\eps)-\pi^*S(\omega)-\eps^2h_S\). On a weighted chart of scale \(\rho\), by the definition of weighted H\"older norm,
\[
  |\phi|+\rho|\nabla\phi|
  \le C\norm{\phi}_{C^{4,\alpha}_\delta}\rho^\delta,
  \qquad
  \rho|\nabla e_\eps|
  +\rho^2|\nabla^2e_\eps|
  \le C\norm{e_\eps}_{C^{2,\alpha}_{\delta-4}}\rho^{\delta-4}.
\]
It follows from \eqref{k2:eq:D-operator} that,
\[
  \rho^{4-\delta}
  \norm{
    \mathscr D_{\omega_\eps,e_\eps}(\phi)
  }_{
    C^{0,\alpha}
    (B_{\omega_\eps}(x,c\rho),\rho^{-2}\omega_\eps)
  }
  \le
  C\norm{\phi}_{C^{4,\alpha}_\delta}
   \norm{e_\eps}_{C^{2,\alpha}_{\delta-4}}
   \rho^{\delta-2}.
\]
Since \(\rho=\tau_\eps(x)\ge c\eps\) and \(\delta-2<0\), taking the supremum over \(x\) gives
\[
  \norm{
    \mathscr D_{\omega_\eps,e_\eps}(\phi)
  }_{C^{0,\alpha}_{\delta-4}}
  \le
  C\eps^{\delta-2}
  \norm{\phi}_{C^{4,\alpha}_\delta}
  \norm{e_\eps}_{C^{2,\alpha}_{\delta-4}}.
\]
Finally, \eqref{k2:eq:scalar-error} gives
\[
  \norm{e_\eps}_{C^{2,\alpha}_{\delta-4}}\le Cr_\eps^{4-\delta},
\]
which proves \eqref{k2:eq:remainder-estimate}.
\end{proof}

We can now solve \eqref{eq:nonlinear-target}.

\begin{theorem}
\label{thm:one-step-approximation}
For every sufficiently small \(\eps>0\), the blow-up
\(\widetilde X=\Bl_SX\) admits a K\"ahler metric
\(\widehat\omega_\eps\) satisfying
\begin{equation}\label{eq:one-step-class}
  [\widehat\omega_\eps]
  =
  \pi^*[\omega]-\eps^2[E]
\end{equation}
and
\begin{equation}\label{eq:one-step-scalar}
  \norm{
    S(\widehat\omega_\eps)-\pi^*S(\omega)
  }_{C^0(\widetilde X)}
  \longrightarrow0.
\end{equation}
\end{theorem}

\begin{proof}
Assume first that \(k>2\), and let \(\Theta_\eps=r_\eps^{3-\delta}+\eps^{2k-2}r_\eps^{4-\delta-2k}\). In this case, equation \eqref{eq:nonlinear-target} is equivalent to
\[
  \widetilde{\mathcal P}_\eps\phi
  =
  \pi^*S(\omega)-S(\omega_\eps)-Q_\eps(\phi).
\]
Define
\begin{equation}\label{gt:eq:fixed-map}
  \mathcal N_\eps(\phi)
  =
  \widetilde{\mathcal P}_\eps^{-1}\bigl(
    \pi^*S(\omega)-S(\omega_\eps)-Q_\eps(\phi)
  \bigr).
\end{equation}
For a constant \(A>0\), let
\[
  \mathcal B_\eps
  =
  \left\{
    \phi\in C^{4,\alpha}_\delta(\widetilde X):
    \norm{\phi}_{C^{4,\alpha}_\delta}
    \le A\Theta_\eps
  \right\}.
\]
For \(\phi\in\mathcal B_\eps\),
\eqref{gt:eq:scalar-error-estimate},
\eqref{gt:eq:inverse-bound}, \eqref{eq:quadratic-single}, and
\eqref{eq:Cdelta-C2} give
\begin{align*}
  \norm{\mathcal N_\eps(\phi)}_{C^{4,\alpha}_\delta}
  &\le
  C\Theta_\eps
  +
  C\norm{\phi}_{C^{4,\alpha}_2}
   \norm{\phi}_{C^{4,\alpha}_\delta}\\
  &\le
  \left(
    C+CA^2\eps^{\delta-2}\Theta_\eps
  \right)\Theta_\eps.
\end{align*}
Choose \(A>2C\). By \eqref{gt:eq:Theta-limits}, \(\mathcal N_\eps\) maps \(\mathcal B_\eps\) into itself for all sufficiently small \(\eps\). For \(\phi,\psi\in\mathcal B_\eps\), \eqref{eq:quadratic-difference}, \eqref{eq:Cdelta-C2}, and \eqref{gt:eq:inverse-bound} similarly give
\[
  \norm{
    \mathcal N_\eps(\phi)-\mathcal N_\eps(\psi)
  }_{C^{4,\alpha}_\delta}
  \le
  CA\eps^{\delta-2}\Theta_\eps
  \norm{\phi-\psi}_{C^{4,\alpha}_\delta}.
\]
The coefficient tends to zero by \eqref{gt:eq:Theta-limits}. Hence \(\mathcal N_\eps\) is a contraction map and has a fixed point \(\phi_\eps\in\mathcal B_\eps\), satisfying
\begin{equation}\label{gt:eq:main-phi-bound}
  \norm{\phi_\eps}_{C^{4,\alpha}_\delta}
  \le A\Theta_\eps,
  \qquad
  \norm{\phi_\eps}_{C^{4,\alpha}_2}
  \le
  CA\eps^{\delta-2}\Theta_\eps
  \longrightarrow0.
\end{equation}
Consequently, \(\widehat\omega_\eps=\omega_\eps+\ddbar\phi_\eps\) is K\"ahler and has the class in \eqref{eq:one-step-class}, by \eqref{gt:eq:class}. Equation \eqref{eq:nonlinear-target} then gives
\begin{equation}\label{gt:eq:exact-scalar}
  S(\widehat\omega_\eps)
  =
  \pi^*S(\omega)
  +
  \sum_{j=1}^N\phi_\eps(q_j)\psi_j.
\end{equation}
Because the points \(q_j\) are away from \(S\), by \eqref{gt:eq:Theta-limits}
\[
  |\phi_\eps(q_j)|
  \le
  C\norm{\phi_\eps}_{C^0_\delta}
  \le
  CA\Theta_\eps
  \longrightarrow0.
\]
Since \(\psi_j\) are fixed and smooth, \eqref{gt:eq:exact-scalar} proves \eqref{eq:one-step-scalar} when \(k>2\).

Assume now that \(k=2\). Again, let \(e_\eps =S(\omega_\eps)-\pi^*S(\omega)-\eps^2h_S\). Equations \eqref{k2:eq:adjoint-difference} and \eqref{eq:case-operator} give
\begin{equation}\label{k2:eq:proof-splitting}
  L_{\omega_\eps}
  =
  \mathcal P_\eps+\mathscr D_{\omega_\eps,e_\eps}.
\end{equation}
Using \eqref{eq:scalar-expansion} and
\eqref{k2:eq:proof-splitting}, the codimension-two equation in
\eqref{eq:nonlinear-target} is equivalent to
\begin{equation}\label{k2:eq:fixed-equation}
  \widetilde{\mathcal P}_\eps\phi
  =
  -e_\eps-Q_\eps(\phi)-\mathscr D_{\omega_\eps,e_\eps}\phi.
\end{equation}
Define
\begin{equation}\label{k2:eq:fixed-map}
  \mathcal N_\eps(\phi)
  =
  \widetilde{\mathcal P}_\eps^{-1}\bigl(
    -e_\eps-Q_\eps(\phi)-\mathscr D_{\omega_\eps,e_\eps}\phi
  \bigr).
\end{equation}
Choose \(C_0>0\). Let \(A_\eps=C_0\eps^\delta r_\eps^{4-\delta}\), and let
\[
  \mathcal B_\eps
  =
  \left\{
    \phi\in C^{4,\alpha}_\delta(\widetilde X):
    \norm{\phi}_{C^{4,\alpha}_\delta}\le A_\eps
  \right\}.
\]
For every \(\phi\in\mathcal B_\eps\),
\eqref{eq:Cdelta-C2} gives
\begin{equation}\label{k2:eq:Kahler-small}
  \norm{\phi}_{C^{4,\alpha}_2}
  \le
  C\eps^{2\delta-2}r_\eps^{4-\delta}
  \longrightarrow0.
\end{equation}
Indeed, the exponent \(2\delta-2+\beta(4-\delta)\) is larger than the positive exponent in \eqref{k2:eq:nonlinear-exponent} by \(-\delta>0\). Thus \(\omega_\eps+\ddbar\phi\) is K\"ahler for every
\(\phi\in\mathcal B_\eps\) and all sufficiently small \(\eps\).

At \(\phi=0\), \eqref{k2:eq:scalar-error} and \eqref{k2:eq:linear-inverse} give
\[
  \norm{\mathcal N_\eps(0)}_{C^{4,\alpha}_\delta}
  \le
  C\eps^\delta r_\eps^{4-\delta}
  \le
  \frac12A_\eps
\]
after choosing \(C_0\) sufficiently large. For \(\phi,\psi\in\mathcal B_\eps\), \eqref{k2:eq:linear-inverse}, \eqref{eq:quadratic-difference}, \eqref{eq:Cdelta-C2}, and \eqref{k2:eq:remainder-estimate} give
\begin{align}
  \norm{
    \widetilde{\mathcal P}_\eps^{-1}\bigl(Q_\eps(\phi)-Q_\eps(\psi)\bigr)
  }_{C^{4,\alpha}_\delta}
  &\le
  C\eps^{3\delta-2}r_\eps^{4-\delta}
  \norm{\phi-\psi}_{C^{4,\alpha}_\delta},
  \label{k2:eq:fixed-quadratic}\\
  \norm{
    \widetilde{\mathcal P}_\eps^{-1} \mathscr D_{\omega_\eps,e_\eps}(\phi-\psi)
  }_{C^{4,\alpha}_\delta}
  &\le
  C\eps^{2\delta-2}r_\eps^{4-\delta}
  \norm{\phi-\psi}_{C^{4,\alpha}_\delta}.
  \label{k2:eq:fixed-remainder}
\end{align}
The coefficient in \eqref{k2:eq:fixed-quadratic} tends to zero by \eqref{k2:eq:nonlinear-exponent}. The exponent in \eqref{k2:eq:fixed-remainder} is larger than the exponent in \eqref{k2:eq:fixed-quadratic} by \(-\delta>0\), so its coefficient also tends to zero. Consequently, \(\mathcal N_\eps\) is a contraction map for sufficiently small \(\eps\). Together with the estimate at \(\phi=0\), this shows that \(\mathcal N_\eps\) maps \(\mathcal B_\eps\) into itself.

Let \(\phi_\eps\in\mathcal B_\eps\) be the unique fixed point. Then
\begin{equation}\label{k2:eq:main-phi-bound}
  \norm{\phi_\eps}_{C^{4,\alpha}_\delta}
  \le
  C\eps^\delta r_\eps^{4-\delta},
  \qquad
  \norm{\phi_\eps}_{C^{4,\alpha}_2}
  \longrightarrow0.
\end{equation}
Thus \(\widehat\omega_\eps=\omega_\eps+\ddbar\phi_\eps\) is K\"ahler and has the class in \eqref{eq:one-step-class}, by \eqref{k2:eq:class}. Finally, \eqref{eq:scalar-expansion}, \eqref{k2:eq:proof-splitting}, and \eqref{k2:eq:fixed-equation} give
\begin{align}
  S(\widehat\omega_\eps)
  &=
  S(\omega_\eps)
  +\mathcal P_\eps\phi_\eps
  +\mathscr D_{\omega_\eps,e_\eps}\phi_\eps
  +Q_\eps(\phi_\eps)\notag\\
  &=
  \pi^*S(\omega)+\eps^2h_S
  +\sum_{j=1}^N\phi_\eps(q_j)\psi_j.
  \label{k2:eq:exact-scalar}
\end{align}
Since the points \(q_j\) are away from \(S\), by \eqref{k2:eq:error-limits},
\[
  |\phi_\eps(q_j)|
  \le
  C\norm{\phi_\eps}_{C^0_\delta}
  \le
  C\eps^\delta r_\eps^{4-\delta}
  \longrightarrow0.
\]
Because \(h_S\) and the \(\psi_j\) are fixed smooth functions, \eqref{k2:eq:exact-scalar} proves
\eqref{eq:one-step-scalar} when \(k=2\). This completes the proof.
\end{proof}

As a direct consequence, we have the following.
\begin{corollary}\label{cor:sign}
Let \((X,\omega)\) be a compact K\"ahler manifold with \(S(\omega)>0\) (respectively \(<0\)). Then $\widetilde X=\Bl_SX$ admits a K\"ahler metric \(\widetilde \omega\) in the class $\pi^*[\omega]-\eps^2[E]$ with \(S(\widetilde \omega)>0\) (respectively \(<0\)) for every $\eps>0$ sufficiently small.
\end{corollary}

\subsection{Iterated approximation for graph resolutions}
\label{sec:iterated-approximation}

We now iterate Theorem~\ref{thm:one-step-approximation} along the blow-ups appearing in a resolution of the graph of a meromorphic map. More precisely, we consider a finite sequence of nontrivial blow-ups
along smooth complex centres of codimension at least two.  We first show the following local smooth convergence away from the exceptional divisor that follows from the approximation construction.

\begin{lemma}
\label{lem:one-step-local-smooth}
Let $\pi:\widetilde  X\to X$ be the blow-up of a compact K\"ahler manifold $(X,\omega)$ along a smooth center of complex codimension at least two, and let $\widehat\omega_\eps$ be the metrics constructed in Theorem \ref{thm:one-step-approximation}.  Then, for every compact set $K\Subset\widetilde X\setminus\Exc(\pi)$, we have
\[
 \widehat\omega_\eps\longrightarrow \pi^*\omega
 \qquad\text{in }C^\infty(K).
\]
\end{lemma}

\begin{proof}
Choose \(K\Subset K'\Subset\widetilde X\setminus\Exc(\pi)\), and let \(\widehat\omega_\eps=\omega_\eps+\ddbar\phi_\eps\),
where $\omega_\eps$ is the approximate metric constructed in Section \ref{sec:approximate}. If the center has codimension greater than two, Proposition~\ref{gt:prop:approximate-metric} \textup{(i)} gives \(\omega_\eps=\pi^*\omega\) on \(K'\) for all sufficiently small \(\eps\). In codimension two, \(K'\) lies in the exterior region for all sufficiently small \(\eps\), and \eqref{k2:eq:exterior-approximate-metric} gives
\[
 \omega_\eps=\pi^*\omega+\eps^2\ddbar(\pi^*\Gamma)
 \qquad\text{on }K'.
\]
Since \(\Gamma\) is smooth away from the center, in both cases
\begin{equation}\label{eq:local-background-convergence}
 \omega_\eps\longrightarrow \pi^*\omega
 \qquad\text{in }C^\infty(K').
\end{equation}

By \eqref{eq:weight}, the weight \(\tau_\eps\) is bounded above and below by positive constants on \(K'\), independently of small \(\eps\).  Hence \eqref{eq:weighted-norm}, together with \eqref{gt:eq:main-phi-bound} and \eqref{gt:eq:Theta-limits} when \(k>2\), and with \eqref{k2:eq:main-phi-bound} and
\eqref{k2:eq:error-limits} when \(k=2\), gives
\[
 \phi_\eps\longrightarrow0
 \qquad\text{in }C^{4,\alpha}(K').
\]

It follows from \eqref{gt:eq:exact-scalar} and \eqref{k2:eq:exact-scalar} that
\begin{equation}\label{eq:one-step-local-scalar}
 S(\widehat\omega_\eps)-\pi^*S(\omega)
 =
 \begin{cases}
 \displaystyle\sum_{j=1}^N\phi_\eps(q_j)\psi_j,
     & k>2,\\[6pt]
 \displaystyle\eps^2h_S+
     \sum_{j=1}^N\phi_\eps(q_j)\psi_j,
     & k=2.
 \end{cases}
\end{equation}
The points \(q_j\) lie in a fixed compact subset away from the center. Equation~\eqref{eq:weighted-norm}, together with \eqref{gt:eq:main-phi-bound} and \eqref{gt:eq:Theta-limits} when \(k>2\),
and with \eqref{k2:eq:main-phi-bound} and \eqref{k2:eq:error-limits} when \(k=2\), gives \(\phi_\eps(q_j)\to0\).  Since \(h_S\) and the \(\psi_j\) are fixed smooth functions, the right-hand side of \eqref{eq:one-step-local-scalar} tends to zero in every \(C^p(K')\) norm.

In local holomorphic coordinates, the coefficient of the fourth derivatives of \(\phi_\eps\) in \(S(\omega_\eps+\ddbar\phi_\eps)\) is \(-(g_{\eps,\phi_\eps})^{i\bar j}(g_{\eps,\phi_\eps})^{p\bar q}\). Here \(g_{\eps,\phi_\eps}\) is the Hermitian matrix of \(\widehat\omega_\eps\). Thus \eqref{eq:local-background-convergence} and the \(C^{4,\alpha}(K')\)-convergence of \(\phi_\eps\) make \eqref{eq:one-step-local-scalar} uniformly elliptic on \(K'\). Elliptic bootstrapping then yields \(\phi_\eps\to0\) in \(C^\infty(K)\). Together with \eqref{eq:local-background-convergence}, this proves the lemma.
\end{proof}

\begin{proof}[Proof of Theorem~\ref{thm:iterated-approximation}]
Set $\omega_{0,\ell}:=\omega$.  Choose an exhaustion
\[
 K_1\Subset K_2\Subset\cdots\Subset
 X_J\setminus\operatorname{Exc}(\mu),
 \qquad
 \bigcup_{\ell\ge1}K_\ell
 =X_J\setminus\operatorname{Exc}(\mu).
\]
For $0\le j\le J$, let
\[
 \rho_j:=\mu_{j+1}\circ\cdots\circ\mu_J:X_J\longrightarrow X_j,
\]
where $\rho_J$ is the identity. For \(1\le j\le J\), denote by \(E_j\subset X_j\) the exceptional
divisor introduced by the blow-up \(\mu_j\), and by \(\widetilde E_j\) its total transform on \(X_J\). Since \(K_\ell\cap\operatorname{Exc}(\mu)=\varnothing\), the set \(\rho_j(K_\ell)\) is a compact subset of
\(X_j\setminus E_j\).

Fix $\ell$. On \(K_\ell\), we use the \(C^p\) norms with respect to the K\"ahler metric \(\mu^*\omega\). Starting with $j=1$ and proceeding to $j=J$, apply Theorem \ref{thm:one-step-approximation} to \((X_{j-1},\omega_{j-1,\ell})\) and the blow-up \(\mu_j:X_j\to X_{j-1}\).  Once \(\eps_{1,\ell},\ldots,\eps_{j-1,\ell}\) have been fixed, the metric \(\omega_{j-1,\ell}\) is fixed.  Therefore \eqref{eq:one-step-scalar} and Lemma~\ref{lem:one-step-local-smooth}, applied on \(\rho_j(K_\ell)\), allow us to choose
\(0<\eps_{j,\ell}<\ell^{-1}\) so that the resulting metric \(\omega_{j,\ell}\) satisfies
\begin{align}
 \norm{S(\omega_{j,\ell})
 -\mu_j^*S(\omega_{j-1,\ell})}_{C^0(X_j)}
 &\le\frac1{J\ell},                                      
 \label{eq:stage-scalar-error}\\
 \norm{\rho_j^*\omega_{j,\ell}
 -\rho_{j-1}^*\omega_{j-1,\ell}}
 _{C^\ell(K_\ell)}
 &\le\frac1{J\ell}.                                     
 \label{eq:stage-metric-error}
\end{align}
Here \(\rho_{j-1}=\mu_j\circ\rho_j\). Write $\omega_\ell:=\omega_{J,\ell}$. Pulling \eqref{eq:stage-scalar-error} to \(X_J\) gives 
\[
 S(\omega_\ell)-\mu^*S(\omega)
 =\sum_{j=1}^J\rho_j^*
 \bigl(S(\omega_{j,\ell})
 -\mu_j^*S(\omega_{j-1,\ell})\bigr).
\]
Since each \(\rho_j\) is surjective, pullback preserves the \(C^0\) norm of a function.  Hence \eqref{eq:stage-scalar-error} implies
\[
 \norm{S(\omega_\ell)-\mu^*S(\omega)}_{C^0(X_J)}
 \le\frac1\ell.
\]
This proves \eqref{eq:intro-iterated-scalar}.  Similarly, on \(K_\ell\),
\[
 \omega_\ell-\mu^*\omega
 =\sum_{j=1}^J
 \bigl(\rho_j^*\omega_{j,\ell}
 -\rho_{j-1}^*\omega_{j-1,\ell}\bigr),
\]
so \eqref{eq:stage-metric-error} gives
\[
 \norm{\omega_\ell-\mu^*\omega}_{C^\ell(K_\ell)}
 \le\frac1\ell.
\]
If \(K\Subset X_J\setminus\operatorname{Exc}(\mu)\) and \(p\ge0\), then \(K\subset K_\ell\) and \(p\le\ell\) for all sufficiently large \(\ell\). The last estimate therefore proves \eqref{eq:intro-iterated-smooth}.

Finally, \eqref{eq:one-step-class} gives, at the \(j\)-th step,
\[
 [\omega_{j,\ell}]
 =\mu_j^*[\omega_{j-1,\ell}]-\eps_{j,\ell}^2[E_j].
\]
Pulling these identities back to \(X_J\), using \([\widetilde E_j]=\rho_j^*[E_j]\), and summing gives
\[
 [\omega_\ell]
 =\mu^*[\omega]-\sum_{j=1}^J
 \eps_{j,\ell}^2[\widetilde E_j].
\]
Since $\eps_{j,\ell}<\ell^{-1}$ for every $j$, the parameter vectors tend to zero. This completes the proof.
\end{proof}

\section{The weighted level-set method}
\label{sec:level-set-method}

In this section, we develop the level-set method for K\"ahler fibrations whose base has pseudo-effective canonical bundle.

\subsection{Holomorphic fibrations and vertical geometry}
Let \(\pi:X^{m+r}\longrightarrow Y^m\) be a surjective holomorphic map between compact K\"ahler manifolds.  Assume that there is a proper analytic subset \(\Delta\subset Y\) such that, with
\[
        Y^\circ:=Y\setminus\Delta,
        \qquad X^\circ:=\pi^{-1}(Y^\circ),
\]
the restriction
\[
        \pi^\circ:=\pi|_{X^\circ}:X^\circ\longrightarrow Y^\circ
\]
is a holomorphic submersion. For \(y\in Y^\circ\), set
\[
        F_y:=\pi^{-1}(y), \qquad \omega_y:=\omega|_{F_y}
\]
and regard \(F_y\) as endowed with the induced K\"ahler metric \(\omega_y\). Since \(\Delta\) is a proper analytic subset, both \(\Delta\) and \(X\setminus X^\circ\) have measure zero.  Therefore all integral identities below may be computed over the regular locus.

On \(X^\circ\), define the vertical tangent bundle
\[
        V:=\ker \partial\pi^\circ\subset T^{1,0}X|_{X^\circ},
\]
and let
\(N:=V^{\perp_\omega}\)
be the \(\omega\)-orthogonal complement.  Thus
\[
        T^{1,0}X|_{X^\circ}=V\oplus N.
\]
Let \(e_1,\ldots,e_r,e_{r+1},\ldots,e_{m+r}\) be an \(\omega\)-unitary frame adapted to this splitting, so that \(e_1,\ldots,e_r\) span \(V\), while \(e_{r+1},\ldots,e_{m+r}\) span \(N\).  We define
\[
        \tr_V\Ric(\omega)
        :=
        \sum_{i=1}^r
        \Ric(\omega)(e_i,\overline{e_i}),
\]
and
\[
        \tr_N\Ric(\omega)
        :=
        \sum_{\alpha=1}^m
        \Ric(\omega)(e_{r+\alpha},\overline{e_{r+\alpha}}).
\]
Then, pointwise on \(X^\circ\),
\begin{equation}\label{eq:trace-splitting-final}
        S(\omega)
        =
        \tr_V\Ric(\omega)+\tr_N\Ric(\omega).
\end{equation}

Since \(\pi^\circ\) is a holomorphic submersion, along every regular fibre \(F_y\) there is a short exact sequence
\[
        0\longrightarrow T^{1,0}F_y
        \longrightarrow T^{1,0}X|_{F_y}
        \longrightarrow T^{1,0}X|_{F_y}/T^{1,0}F_y:=N_{F_y/X}
        \longrightarrow 0 .
\]
Taking the dual sequence and then taking the determinant gives
\[
        K_X|_{F_y}
        =
        K_{F_y}\otimes (\det N_{F_y/X})^* .
\]
Since \(N_{F_y/X}\) is trivial, we have
\[
        c_1(X)|_{F_y}=c_1(F_y).
\]
Therefore
\[
        [\Ric(\omega)|_{F_y}]
        =
        2\pi c_1(X)|_{F_y}
        =
        2\pi c_1(F_y)
        =
        [\Ric(\omega_y)].
\]
By the \(\partial\bar\partial\)-lemma, the difference
\[
        \Ric(\omega)|_{F_y}-\Ric(\omega_y)
\]
is \(\sqrt{-1}\partial\bar\partial\)-exact. Wedging with \(\omega_y^{r-1}\) and using Stokes' theorem, we obtain
\begin{equation}\label{eq:vertical-trace-fibre-scalar-final}
        \int_{F_y}\tr_V\Ric(\omega)\,\omega_y^r
        =
        \int_{F_y}S(\omega_y)\,\omega_y^r.
\end{equation}

The complex differential of \(\pi\) is a holomorphic bundle morphism
\[
        \partial\pi:T^{1,0}X\longrightarrow \pi^*T^{1,0}Y.
\]
Taking its \(m\)-th exterior power gives
\[
        \Lambda^m\partial\pi:
        \Lambda^mT^{1,0}X
        \longrightarrow
        \pi^*\Lambda^mT^{1,0}Y.
\]
Equivalently, it defines a holomorphic section
\[
        s_\pi:=\Lambda^m\partial\pi
        \in
        \Gamma\left(
        \Lambda^m(T^{1,0}X)^*
        \otimes
        \pi^*\Lambda^mT^{1,0}Y
        \right).
\]
We define the complex Jacobian of \(\pi\) by
\[
        J_\pi:=|s_\pi|^2_{\omega,\eta}.
\]
At a point \(x\in X^\circ\), choose the adapted \(\omega\)-unitary frame \(e_1,\ldots,e_r,e_{r+1},\ldots,e_{m+r}\) as above.  Choose an \(\eta\)-unitary frame \(\varepsilon_1,\ldots,\varepsilon_m\) of \(T^{1,0}_{\pi(x)}Y\) such that
\[
        \partial\pi(e_{r+\alpha})
        =
        \lambda_\alpha\varepsilon_\alpha,
        \qquad
        \lambda_\alpha>0,
        \qquad
        \alpha=1,\ldots,m.
\]

Let \((\mathcal E,h)\) be a Hermitian holomorphic vector bundle over a K\"ahler manifold \((Z,\omega)\), and let \(\Theta(\mathcal E,h)\) be its Chern curvature.  If \(s\) is a holomorphic section of \(\mathcal E\), then
\begin{equation}\label{eq:chern-bochner}
        \Delta_\omega |s|_h^2
        =
        |\nabla s|^2
        -
        \left\langle
        \sqrt{-1}\tr_\omega\Theta(\mathcal E,h)s,s
        \right\rangle_h .
\end{equation}
This is the standard Chern-Bochner identity for holomorphic sections; see, for example, \cite[eq. (4)]{KobayashiWu70}.

\begin{lemma}\label{lem:jac-bochner}
On \(X^\circ\), one has
\begin{equation}\label{eq:jac-bochner}
        \Delta_\omega J_\pi
        =
        |\nabla s_\pi|^2
        +
        J_\pi\,\tr_N\Ric(\omega)
        -
        J_\pi\,\tr_\omega\pi^*\Ric(\eta).
\end{equation}
\end{lemma}

\begin{proof}
Apply \eqref{eq:chern-bochner} to the holomorphic section
\[
        s_\pi\in\Gamma(\mathcal E),
        \qquad
        \mathcal E=
        \Lambda^m(T^{1,0}X)^*
        \otimes
        \pi^*\Lambda^mT^{1,0}Y.
\]
Let \(\theta^1,\dots,\theta^{m+r}\) be the coframe dual to \(e_1,\dots,e_{m+r}\).  At \(x\),
\[
        s_\pi
        =
        \lambda_1\cdots\lambda_m\,
        \theta^{r+1}\wedge\cdots\wedge\theta^{m+r}
        \otimes
        \varepsilon_1\wedge\cdots\wedge\varepsilon_m .
\]
Hence
\[
        J_\pi=|s_\pi|^2=\lambda_1^2\cdots\lambda_m^2.
\]

The curvature of \(\mathcal E\) is the sum of the curvatures of the two tensor factors.  First, the curvature of \((T^{1,0}X)^*\) is the negative transpose of the curvature of \(T^{1,0}X\).  Therefore, one has
\[
        \left\langle
        \sqrt{-1}\tr_\omega
        \Theta\bigl(\Lambda^m(T^{1,0}X)^*\bigr)s_\pi,
        s_\pi
        \right\rangle
        =
        -J_\pi
        \sum_{\alpha=1}^m
        \Ric(\omega)(e_{r+\alpha},\overline{e_{r+\alpha}}).
\]
Second, \(\Lambda^mT^{1,0}Y=K_Y^{-1}\), and the Chern curvature of \(K_Y^{-1}\) is represented by the Ricci form \(\Ric(\eta)\).  Pulling back to \(X\) and contracting with \(\omega\), we get
\[
        \left\langle
        \sqrt{-1}\tr_\omega
        \Theta\bigl(\pi^*\Lambda^mT^{1,0}Y\bigr)s_\pi,
        s_\pi
        \right\rangle
        =
        J_\pi
        \sum_{i=1}^{m+r}
        \Ric(\eta)\bigl(
        \partial\pi(e_i),
        \overline{\partial\pi(e_i)}
        \bigr).
\]
Since \(\partial\pi(e_i)=0\) for \(1\le i\le r\), and \( \partial\pi(e_{r+\alpha}) = \lambda_\alpha\varepsilon_\alpha\), this becomes
\[
        J_\pi
        \sum_{\alpha=1}^m
        \lambda_\alpha^2
        \Ric(\eta)(\varepsilon_\alpha,\overline{\varepsilon_\alpha}).
\]
Substituting the two curvature contributions into \eqref{eq:chern-bochner} gives \eqref{eq:jac-bochner}.
\end{proof}

We shall also use the following coarea formula for the holomorphic submersion \(\pi^\circ:(X^\circ,\omega)\to(Y^\circ,\eta)\).

\begin{lemma}
For every smooth function \(\Phi\in C^\infty(X)\),
\begin{equation}\label{eq:coarea-final}
        \int_X J_\pi\,\Phi\,\omega^{m+r}
        =
        \frac{(m+r)!}{r!\,m!}
        \int_{Y^\circ}
        \left(
        \int_{F_y}\Phi\,\omega_y^r
        \right)
        \eta^m .
\end{equation}
\end{lemma}
\begin{proof}
On \(X^\circ\), choose adapted unitary frames as above.  At the point under consideration one has
\[
J_\pi\,\omega^{m+r}
=
\frac{(m+r)!}{m!\,r!}\,
\omega_y^r\wedge\pi^*\eta^m.
\]
Indeed, the horizontal eigenvalues of \(\pi^*\eta\) with respect to \(\omega\) are \(\lambda_1^2,\ldots,\lambda_m^2\), whose product is \(J_\pi\).  Integrating this identity first along the fibres and then over \(Y^\circ\) proves \eqref{eq:coarea-final}, since the critical locus of \(\pi\) and \(Y\setminus Y^\circ\) have measure zero.
\end{proof}
\subsection{The weighted inequality on the resolved base}
\label{subsec:weighted-level-set}

Continue with the preceding notation, assume that \(Y^\circ\) is connected and that \(K_Y\) is pseudo-effective, and fix a K\"ahler metric \(\eta\) on \(Y\).  For \(z\in Y^\circ\), write \(F_z:=\pi^{-1}(z)\) and \(\omega_z:=\omega|_{F_z}\).  Let \(h_{K,\eta}\) be the Hermitian metric on \(K_Y\) induced by \(\eta\), so that
\[
        \sqrt{-1}\Theta(K_Y,h_{K,\eta})
        =-\Ric(\eta).
\]
Since \(K_Y\) is pseudo-effective, there is a quasi-psh function \(\psi\not\equiv-\infty\) normalized by \(\sup_Y\psi=0\) such that the singular Hermitian metric \(h_K=e^{-\psi}h_{K,\eta}\) has semipositive curvature
\begin{equation}\label{eq:psef-current-KY-final}
        T:=\sqrt{-1}\Theta(K_Y,h_K)
        =
        -\Ric(\eta)
        +\sqrt{-1}\partial\bar\partial\psi
        \ge0
\end{equation}
as a closed positive \((1,1)\)-current.

\begin{proposition}
\label{prop:psef-fibrewise-scalar-final}
With \(\eta\) and \(\psi\) fixed as above, one has
\begin{equation}\label{eq:psef-fibrewise-scalar-final}
        \int_{Y^\circ}e^\psi
        \left(
        \int_{F_z}S(\omega)\,\omega_z^r
        \right)
        \eta^m
        \le
        \int_{Y^\circ}e^\psi
        \left(
        \int_{F_z}S(\omega_z)\,\omega_z^r
        \right)
        \eta^m.
\end{equation}
\end{proposition}

\begin{proof}
By Demailly's regularization theorem \cite[Proposition~3.1]{Demailly92}, there are quasi-psh functions
\(\psi_j\le0\) normalized by \(\sup_Y \psi_j=0\), with analytic singularities, and numbers \(\delta_j\downarrow0\) such that
\[
\psi_j\longrightarrow\psi
\quad\text{in }L^1(Y)\text{ and almost everywhere on }Y.
\]
Let \(A_j\) be the polar set of \(\psi_j\) which is a proper analytic subset. Then \(\psi_j\) is smooth on \(Y\setminus A_j\), and
\begin{equation}\label{eq:regularized-psef-current-final}
        T_j
        :=
        -\Ric(\eta)
        +\sqrt{-1}\partial\bar\partial\psi_j
        \ge-\delta_j\eta.
\end{equation}
Define
\[
        \widetilde T_j:=T_j+\delta_j\eta\ge0,
        \qquad
        u_j:=e^{\psi_j\circ\pi}J_\pi.
\]
The function \(u_j\) is the squared norm of \(s_\pi=\Lambda^m\partial\pi\) with respect to the metric on
\(\Lambda^m(T^{1,0}X)^* \otimes\pi^*K_Y^{-1}\) induced by \(\omega\) and \(e^{\psi_j}h_{K,\eta}^{-1}\). It follows from \eqref{eq:trace-splitting-final}, \eqref{eq:jac-bochner} and \eqref{eq:regularized-psef-current-final} that, on \(X^\circ\setminus\pi^{-1}(A_j)\),
\begin{equation}\label{eq:weighted-bochner-scalar-final}
\begin{aligned}
        \Delta_\omega u_j
        =
        |\nabla^j s_\pi|^2
        +
        u_j\left(
        S(\omega)-\tr_V\Ric(\omega)
        +\tr_\omega\pi^*\widetilde T_j
        \right) 
        -\delta_j u_j\tr_\omega\pi^*\eta,
\end{aligned}
\end{equation}
where \(\nabla^j\) is the associated Chern connection. We next remove the analytic singularities. Locally,
\[
        \psi_j
        =
        c_j\log\!\left(\sum_\nu|f_{j,\nu}|^2\right)+v_j,
        \qquad c_j>0,
\]
where \(v_j\) is smooth. Thus \(e^{\psi_j/2}\), extended by zero across \(A_j\), belongs locally to \(W^{1,2}\). Fix \(C_0\ge e\) and let \(\rho_j:=\psi_j-C_0\). Then \(\rho_j\le-e\), and \(\rho_j\) and \(\psi_j\) have the same analytic singularities. Let \(\chi\in C^\infty(\mathbb R,[0,1])\) satisfy
\[
        \chi=1\quad\text{on }(-\infty,1],
        \qquad
        \chi=0\quad\text{on }[2,\infty),
        \qquad
        \chi'\le0.
\]
Define
\[
        \widehat\chi_{j,R}
        :=
        \chi\!\left(
        R^{-1}\log\bigl(-\rho_j\circ\pi\bigr)
        \right),
\]
extended by zero across \(\pi^{-1}(A_j)\). For each fixed \(R\), this extension is smooth because the cutoff function vanishes identically in a neighbourhood of \(\pi^{-1}(A_j)\). Moreover, \(\widehat\chi_{j,R}\to1\) almost everywhere as \(R\to\infty\), and
\[
        |d\widehat\chi_{j,R}|
        \le
        \frac{C}{R}
        \frac{|d(\rho_j\circ\pi)|}{-\rho_j\circ\pi}
        \mathbf 1_{\{
        e^R\le-\rho_j\circ\pi\le e^{2R}\}}.
\]

Apply \cite[Theorem~1.10]{BierstoneMilman97} to the analytic singularities of \(\rho_j\circ\pi\). This gives a finite composition of blow-ups along smooth centres \(\tau_j:\widetilde X_j\longrightarrow X\), such that the pulled-back singularities have simple normal crossings. Thus, locally on \(\widetilde X_j\),
\[
        \widetilde\rho_j
        :=
        \rho_j\circ\pi\circ\tau_j
        =
        \sum_{\ell=1}^k
        a_{j,\ell}\log|z_\ell|^2+b_j,
        \qquad a_{j,\ell}>0,
\]
where \(b_j\) is smooth. Denote \(\widetilde\chi_{j,R}:=\widehat\chi_{j,R}\circ\tau_j\). Fix a K\"ahler metric \(\widetilde\omega_j\) on \(\widetilde X_j\). Since \(\widetilde X_j\) is compact, there exists a constant \(C_j>0\), so that \(\tau_j^*\omega\le C_j\widetilde\omega_j\). Using the change-of-variables formula for the degree-one modification \(\tau_j\), we obtain
\[
\begin{aligned}
 \int_X
 |d\widehat\chi_{j,R}|_\omega^2\,\omega^{m+r}                                                  
 &=
 2(m+r)\int_{\widetilde X_j}
 \sqrt{-1}\partial\widetilde\chi_{j,R}
 \wedge\bar\partial\widetilde\chi_{j,R}
 \wedge(\tau_j^*\omega)^{m+r-1}                          \\
 &\le
 C_j\int_{\widetilde X_j}
 |d\widetilde\chi_{j,R}|_{\widetilde\omega_j}^2
 \,\widetilde\omega_j^{m+r}          \\
 &\le
 \frac{C_j}{R^2}
 \int_{\{e^R\le-\widetilde\rho_j\le e^{2R}\}}
 \frac{|d\widetilde\rho_j|_{\widetilde\omega_j}^2}
      {(-\widetilde\rho_j)^2}
 \,\widetilde\omega_j^{m+r}\\
 &\le
 \frac{C_j}{R^2e^R} \longrightarrow0
        \qquad\text{as }R\to\infty.
\end{aligned}
\]
Since \(0\le u_j\le J_\pi\) and \(J_\pi\) is smooth on \(X\),
\begin{equation}\label{eq:weighted-log-log-cutoff-final}
        \int_Xu_j
        |d\widehat\chi_{j,R}|_\omega^2\,\omega^{m+r}
        \longrightarrow0.
\end{equation}
For \(\widehat\chi_{j,R}^2\Delta_\omega u_j\), integrating by parts and using Kato's inequality give
\begin{align*}
 &\left|
        \int_X
        \widehat\chi_{j,R}^2
        \Delta_\omega u_j\,\omega^{m+r}\right|  \le
        \frac12
        \int_X
        \widehat\chi_{j,R}^2
        |\nabla^j s_\pi|^2\,\omega^{m+r}
        +
        C\int_X
        u_j|d\widehat\chi_{j,R}|_\omega^2\,\omega^{m+r}.
\end{align*}
Combining with \eqref{eq:weighted-bochner-scalar-final} and letting \(R\to\infty\) gives
\begin{equation}\label{eq:pre-coarea-weighted-final}
        \int_Xu_j
        \bigl(S(\omega)-\tr_V\Ric(\omega)\bigr)
        \,\omega^{m+r}
        \le C\delta_j,
\end{equation}
where \(C\) is independent of \(j\). Apply the coarea formula to \eqref{eq:pre-coarea-weighted-final}. After absorbing the constant into \(C\), we obtain
\[
        \int_{Y^\circ}e^{\psi_j}
        \left[
        \int_{F_z}S(\omega)\,\omega_z^r
        -
        \int_{F_z}\tr_V\Ric(\omega)\,\omega_z^r
        \right]\eta^m
        \le C\delta_j.
\]
By \eqref{eq:vertical-trace-fibre-scalar-final}, this becomes
\begin{equation}\label{eq:regularized-fibrewise-scalar-final}
        \int_{Y^\circ}e^{\psi_j}
        \left[
        \int_{F_z}S(\omega)\,\omega_z^r
        -
        \int_{F_z}S(\omega_z)\,\omega_z^r
        \right]\eta^m
        \le C\delta_j.
\end{equation}
Let \(\delta_j\to0\).  The expression in square brackets is uniformly bounded on \(Y^\circ\): this follows from the constancy of the fibre volume, the boundedness of \(S(\omega)\), and
\eqref{eq:vertical-trace-fibre-scalar-final}.  Since \(0\le e^{\psi_j}\le1\), dominated convergence gives
\eqref{eq:psef-fibrewise-scalar-final}.
\end{proof}


\section{Systolic inequalities on PSC K\"ahler manifolds}\label{sec:systolic-inequalities-psc}

In this section, we prove the systolic inequalities by the level-set method.

\subsection{The MRC fibration and its graph resolution}

Let \((X^n,\omega)\) be a compact K\"ahler manifold with \(S(\omega)>0\). By the Chern--Weil formula,
\[
        0<\int_X S(\omega)\,\omega^n
        =2\pi n\,c_1(X)\cdot[\omega]^{n-1}.
\]
Thus \(c_1(K_X)\cdot[\omega]^{n-1}<0\).  If \(K_X\) were pseudo-effective, a closed positive current in \(c_1(K_X)\) would have nonnegative pairing with \([\omega]^{n-1}\), a contradiction.  Hence \(K_X\) is not pseudo-effective, and Ou's characterization \cite[Theorem~1.1]{Ou2025} shows that \(X\) is uniruled.

We first recall the maximal rationally connected (MRC) fibration. By Campana's rational quotient construction for compact K\"ahler manifolds \cite{Campana81} (see also \cite[Theorem~2.7]{HoeringPeternell11}, \cite{KMM92} and \cite[Chapter~IV]{Kollar96} in the projective setting), there is an almost holomorphic map
\begin{equation*}
        \varphi:X\dashrightarrow Z
\end{equation*}
with connected fibres, unique up to bimeromorphic equivalence. Here almost holomorphic means that there are nonempty Zariski open sets \(X^0\subset X\) and \(Z^0\subset Z\) such that \(\varphi^0:X^0\longrightarrow Z^0\) is a proper holomorphic map. The general fibre is smooth and rationally connected. Moreover, every rational curve through a very general point of
\(X\) is contained in the corresponding fibre.  This maximality characterizes \(\varphi\).

The base \(Z\) is a normal compact K\"ahler space and is not uniruled. Choose a resolution \(\nu:\widehat Z\longrightarrow Z\) and resolve the main component of the graph of \(\varphi\):
\begin{equation}\label{eq:mrc-graph-final}
        \begin{array}{ccc}
        \widehat X & \xrightarrow{\widehat\pi} & \widehat Z \\
        \mu\downarrow & & \downarrow\nu \\
        X & \dashrightarrow & Z .
        \end{array}
\end{equation}
The manifolds \(\widehat X\) and \(\widehat Z\) may be chosen compact K\"ahler, \(\widehat\pi\) is holomorphic, and \(\mu\) is a composition of blow-ups along smooth centres by \cite[Theorem~1.10]{BierstoneMilman97} for principalization and \cite{Varouchas89} for the K\"ahler property.  After shrinking a nonempty Zariski open set \(\widehat Z^\circ\subset\widehat Z\), the map
\[
        \widehat\pi:\widehat X^\circ
        :=\widehat\pi^{-1}(\widehat Z^\circ)
        \longrightarrow\widehat Z^\circ
\]
is a proper holomorphic submersion and \(\mu\) is biholomorphic on a neighbourhood of \(\widehat X^\circ\).  Thus the regular fibres \(\widehat F_z\) are identified with the general MRC fibres in \(X\), and the exceptional classes of \(\mu\) restrict trivially to them.

Since uniruledness is birationally invariant, \(\widehat Z\) is not uniruled.  By \cite[Theorem~1.1]{Ou2025}, \(K_{\widehat Z}\) is therefore pseudo-effective. Consequently the graph resolution
\eqref{eq:mrc-graph-final} satisfies the hypotheses of Section~\ref{sec:level-set-method}.

\begin{definition}[Rational dimension]\label{def:rational-dimension}
The rational dimension of \(X\), denoted by \(r(X)\), is the complex dimension of a general fibre of its MRC fibration. Equivalently,
\[
        r(X)=\dim X-\dim Z.
\]
\end{definition}
Since \(X\) is uniruled, \(r(X)\ge1\).  Moreover, \(r(X)=n\) if and only if \(X\) is rationally connected.  Since the MRC fibration is unique up to bimeromorphic equivalence, \(r(X)\) is a birational invariant.

\subsection{A sharp even-dimensional systolic inequality}

We now establish a sharp upper bound for every even-dimensional homological systole for fibrations whose general fibre is the projective space.

\begin{proof}[Proof of Theorem \ref{thm:even-sys}]
Choose a smooth compact K\"ahler model \(\nu:\widehat Z\to Z\) and a graph resolution
\[
        \begin{array}{ccc}
        \widehat X & \xrightarrow{\widehat\pi} & \widehat Z \\
        \mu\downarrow & & \downarrow\nu \\
        X & \dashrightarrow & Z .
        \end{array}
\]
By \cite[Theorem~1.10]{BierstoneMilman97}, we may take \(\mu\) to be a finite composition of blow-ups along smooth centres. Theorem~\ref{thm:iterated-approximation} gives K\"ahler metrics \(\omega_\ell\) on \(\widehat X\) such that
\begin{equation}\label{eq:even-sc-approx-final}
        \bigl\|S(\omega_\ell)-\mu^*S(\omega)
        \bigr\|_{C^0(\widehat X)}\longrightarrow0.
\end{equation}

Since \(\widehat Z\) is not uniruled, \cite[Theorem~1.1]{Ou2025} implies that \(K_{\widehat Z}\) is
pseudo-effective.  Let \(e^\psi\) be the weight used in \eqref{eq:psef-fibrewise-scalar-final}.  Over a nonempty Zariski open set \(\widehat Z^\circ\), the map \(\widehat\pi\) is a submersion, \(\mu\) is an
isomorphism along the fibres, and \(\widehat F_z\cong\PP^r\).  Let \(H=c_1(\mathcal O_{\PP^r}(1))\).  Then
\([\omega_\ell|_{\widehat F_z}]=[\mu^*\omega|_{\widehat F_z}]=aH\). Therefore
\[
        \int_{\widehat F_z}\omega_{\ell,z}^r=a^r
\]
and, by the fibrewise Chern--Weil formula,
\begin{equation}\label{eq:fibre-scalar-pr-final}
        \int_{\widehat F_z}S(\omega_{\ell,z})
        \,\omega_{\ell,z}^r
        =2\pi r(r+1)a^{r-1}.
\end{equation}

Applying \eqref{eq:psef-fibrewise-scalar-final} on the graph resolution gives
\[
\begin{aligned}
        \min_{\widehat X}S(\omega_\ell)a^r
        \int_{\widehat Z^\circ}e^\psi\eta^m
        \le
        2\pi r(r+1)a^{r-1}
        \int_{\widehat Z^\circ}e^\psi\eta^m,
\end{aligned}
\]
which is exactly
\[
        \min_{\widehat X}S(\omega_\ell)a
        \le2\pi r(r+1).
\]
Letting \(\ell\to\infty\) in view of \eqref{eq:even-sc-approx-final}, we obtain
\begin{equation}\label{eq:minS-a}
        \min_XS(\omega)a\le2\pi r(r+1).
\end{equation}

Now fix \(1\le q\le r\).  A linear \(\PP^q\subset\widehat F_z\) maps biholomorphically to a complex
\(q\)-fold in \(X\).  Its homology class is nonzero, since
\[
        \int_{\PP^q}\omega^q=a^q>0.
\]
Hence \(\sys_{2q}(\omega)\le a^q\).  Combining this with
\eqref{eq:minS-a} proves \eqref{eq:sharp-even-systolic}.

Suppose now that equality holds in \eqref{eq:sharp-even-systolic}. We then have
\[
        \sys_{2q}(\omega)^{1/q}=a,
        \qquad
        \min_XS(\omega)a=2\pi r(r+1).
\]
Proposition~\ref{prop:sharp-fibre-rigidity} therefore gives the asserted splitting, and a fibrewise linear \(\PP^q\) realizes the systole. The converse is obvious.
\end{proof}

\subsection{The \texorpdfstring{\(2\)}{2}-systole and rational dimension}\label{subsec:rational-dimensional-2-systole}

In this subsection, we prove the global \(2\)-systole estimate.  In particular, the sharp constant is controlled by the rational dimension of the manifold rather than by its complex dimension. We first record the algebraic estimate used on the rationally connected fibres.

\begin{lemma}\label{lem:mori-length-final}
Let \(F^r\) be a smooth rationally connected manifold, and let \(\alpha\in\mathcal K(F)\) be a K\"ahler class.  Define
\[
        m_F(\alpha)
        :=
        \inf\{\alpha\cdot C\mid C\subset F\text{ a rational curve}\}.
\]
Then
\begin{equation}\label{eq:mori-length-final}
        m_F(\alpha)
        \frac{c_1(F)\cdot\alpha^{r-1}}{\alpha^r}
        \le
        r+1.
\end{equation}
Moreover, equality in \eqref{eq:mori-length-final} holds if and only if \(F\cong\mathbb P^r\).
\end{lemma}

\begin{proof}
A rationally connected compact K\"ahler manifold is projective.  Choose an ample integral class \(A\). Since \(\alpha-\tau A\) is K\"ahler for some \(\tau>0\), every curve satisfies
\(\alpha\cdot C\ge\tau A\cdot C\ge\tau\).  Thus \(m_F(\alpha)>0\).

Define
\[
        D_\alpha
        :=
        \frac{r+1}{m_F(\alpha)}\alpha-c_1(F).
\]
We claim that \(D_\alpha\) is nef.  By Mori's cone theorem \cite[Theorem~1.4]{Mori82}, every \(K_F\)-negative extremal ray is generated by a rational curve \(C\) satisfying
\[
        0<-K_F\cdot C\le r+1.
\]
For such a curve,
\[
        D_\alpha\cdot C
        =
        \frac{r+1}{m_F(\alpha)}\alpha\cdot C-c_1(F)\cdot C
        \ge
        r+1-(-K_F\cdot C)
        \ge0.
\]
On the part of the Mori cone where \(K_F\ge0\), the inequality \(D_\alpha\cdot\gamma\ge0\) is immediate because \(\alpha\) is K\"ahler and \(-c_1(F)\cdot\gamma\ge0\).  Hence \(D_\alpha\) is nef.  Therefore
\[
        0
        \le
        D_\alpha\cdot\alpha^{r-1}
        =
        \frac{r+1}{m_F(\alpha)}\alpha^r
        -
        c_1(F)\cdot\alpha^{r-1}.
\]
This proves \eqref{eq:mori-length-final}.

If equality holds, then \(D_\alpha\cdot\alpha^{r-1}=0\).  Since \(D_\alpha\) is nef and \(\alpha\) is K\"ahler, this implies \(D_\alpha\equiv0\). Thus
\[
        c_1(F)=\frac{r+1}{m_F(\alpha)}\alpha>0,
\]
which implies \(F\) is Fano. Moreover every rational curve \(C\subset F\) satisfies
\[
        -K_F\cdot C
        =
        \frac{r+1}{m_F(\alpha)}\alpha\cdot C
        \ge
        r+1.
\]
Hence the pseudoindex of \(F\) is at least \(r+1\).  By the Cho--Miyaoka--Shepherd-Barron characterization
\cite[Theorem~0.1]{CMSB02}, we have \(F\cong\mathbb P^r\). The converse is obvious.
\end{proof}

We now estimate the \(2\)-systole.

\begin{proof}[Proof of Theorem~\ref{thm:2sys}]
Let \(\varphi:X\dashrightarrow Z\) be the MRC fibration of \(X\). Since \(X\) has positive scalar curvature, \(X\) is uniruled, the general MRC fibre is rationally connected, and the base is not uniruled.

First suppose \(r=n\).  Then \(X\) is rationally connected.  Since \(X\) is compact K\"ahler, it is projective.  Applying Lemma~\ref{lem:mori-length-final} to \(F=X\) and \(\alpha=[\omega]\), and using \(\sys_2(\omega)\le m_X([\omega])\) together with
\[
        \min_XS(\omega)
        \le
        \frac{\int_XS(\omega)\,\omega^n}{\int_X\omega^n}
        =
        2\pi n\frac{c_1(X)\cdot[\omega]^{n-1}}{[\omega]^n},
\]
we get
\[
        \min_XS(\omega)\,\sys_2(\omega)
        \le
        2\pi n(n+1),
\]
which is \eqref{eq:rational-dimensional-2sys-final} in this case.

Now assume \(r<n\), and use the graph resolution \eqref{eq:mrc-graph-final}. Theorem~\ref{thm:iterated-approximation} gives K\"ahler metrics \(\omega_\ell\) on \(\widehat X\) such that
\begin{equation}\label{eq:sc-approx-resolution-final}
        \bigl\|S(\omega_\ell)-\mu^*S(\omega)\bigr\|_{C^0(\widehat X)}
        \longrightarrow0.
\end{equation}
In particular, for \(\ell\) sufficiently large, \(\omega_\ell\) still has positive scalar curvature.

We may therefore apply \eqref{eq:psef-fibrewise-scalar-final} to \(\widehat\pi:(\widehat X,\omega_\ell)\to\widehat Z\).  For \(z\in\widehat Z^\circ\), Lemma~\ref{lem:mori-length-final} gives
\[
        \int_{\widehat F_z}S(\omega_{\ell,z})
        \,\omega_{\ell,z}^r
        \le
        \frac{2\pi r(r+1)}{m_{\ell,z}}
        \int_{\widehat F_z}\omega_{\ell,z}^r,
\]
where
\[
        m_{\ell,z}
        :=
        \inf\left\{
        \int_C\omega_\ell
        \;\middle|\;
        C\subset\widehat F_z\text{ a rational curve}
        \right\}.
\]
Since \(\mu\) is an isomorphism along \(\widehat F_z\), every rational curve \(C\subset\widehat F_z\) maps to a rational curve \(\mu(C)\subset X\), and the exceptional divisors of \(\mu\) do not meet \(C\).  The cohomology-class formula in Theorem~\ref{thm:iterated-approximation} therefore gives
\[
        \int_C\omega_\ell
        =
        \int_{\mu(C)}\omega.
\]
The curve \(\mu(C)\) is homologically non-trivial because it is a positive-area holomorphic curve in the K\"ahler manifold \((X,\omega)\). Therefore
\[
        m_{\ell,z}
        \ge
        \sys_2(\omega).
\]
It follows that
\[
        \int_{\widehat F_z}S(\omega_{\ell,z})
        \,\omega_{\ell,z}^r
        \le
        \frac{2\pi r(r+1)}{\sys_2(\omega)}
        \int_{\widehat F_z}\omega_{\ell,z}^r.
\]
Substituting this estimate into \eqref{eq:psef-fibrewise-scalar-final}, with \(\omega\) there replaced by \(\omega_\ell\), gives
\[
        \min_{\widehat X}S(\omega_\ell)
        \int_{\widehat Z^\circ}e^\psi
        \left(
        \int_{\widehat F_z}\omega_{\ell,z}^r
        \right)
        \eta^m
        \le
        \frac{2\pi r(r+1)}{\sys_2(\omega)}
        \int_{\widehat Z^\circ}e^\psi
        \left(
        \int_{\widehat F_z}\omega_{\ell,z}^r
        \right)
        \eta^m.
\]
Consequently,
\[
        \min_{\widehat X}S(\omega_\ell)\,\sys_2(\omega)
        \le
        2\pi r(r+1).
\]
Letting \(\ell\to\infty\) and using
\eqref{eq:sc-approx-resolution-final}, we get
\[
        \min_XS(\omega)\,\sys_2(\omega)
        \le
        2\pi r(r+1).
\]
This proves \eqref{eq:rational-dimensional-2sys-final}.

Now assume the equality holds in \eqref{eq:rational-dimensional-2sys-final}. Let \(\sigma:=\sys_2(\omega)\). If \(r=n\), we have \(S(\omega)\equiv\min_XS(\omega)\) and \(\sigma=m_X([\omega])\) and equality holds in Lemma~\ref{lem:mori-length-final}. Hence \(X\cong\PP^n\). Writing \([\omega]=aH\), we obtain
\[
  a=\sigma,
  \qquad
  \min_XS(\omega)a=2\pi n(n+1).
\]
Applying Proposition~\ref{prop:sharp-fibre-rigidity} to the map \(X\to\{\mathrm{pt}\}\) gives  \((X,\omega)\cong(\PP^n,\sigma\omega_{\mathrm{FS}})\)

Assume next that \(r<n\). Since \(\min_{\widehat X}S(\omega_\ell)\rightarrow\min_XS(\omega)\) by \eqref{eq:sc-approx-resolution-final}, equality in \eqref{eq:rational-dimensional-2sys-final} forces equality throughout the chain of inequalities used above. More precisely, the fibre volume and the fibrewise total scalar curvature depend only on the restricted K\"ahler class, which is independent of \(\ell\). Hence the weighted integral of the nonnegative fibrewise defect
\[
  \frac{2\pi r(r+1)}{\sigma}
  \int_{\widehat F_z}\omega_{\ell,z}^r
  -
  \int_{\widehat F_z}
    S(\omega_{\ell,z})\,\omega_{\ell,z}^r
\]
vanishes. Since \(e^\psi>0\) almost everywhere and this defect is locally constant in \(z\), it vanishes on every general fibre. Consequently, equality holds both in Lemma~\ref{lem:mori-length-final} and in \(m_{\ell,z}\ge\sigma\). Therefore
\[
  \widehat F_z\cong\PP^r,
  \qquad
  m_{\ell,z}=\sigma,
  \qquad
  [\mu^*\omega|_{\widehat F_z}]=\sigma H.
\]
Thus the hypotheses of Proposition~\ref{prop:sharp-fibre-rigidity} are satisfied with \(a=\sigma\), and the proposition gives the asserted splitting.
\end{proof}

\subsection{Equality and rigidity}\label{sec:systolic-rigidity}

We first isolate the equality mechanism common to
Theorems~\ref{thm:even-sys} and~\ref{thm:2sys}.

\begin{proposition}\label{prop:sharp-fibre-rigidity}
Let \(f:X^{m+r}\dashrightarrow Z^m\) be an almost holomorphic fibration from a compact K\"ahler manifold onto a normal compact K\"ahler space which is not uniruled, and assume that a general fibre is biholomorphic to \(\PP^r\). Let \(\omega\) be a K\"ahler metric on \(X\) with positive scalar curvature, and let \(a:=[\omega]\cdot [L]\), where \(L\) is a line in a general fibre.  If
\begin{equation}\label{eq:sharp-relation-rigidity}
        \min_XS(\omega)\,a=2\pi r(r+1),
\end{equation}
then the universal cover splits holomorphically and isometrically as
\begin{equation}\label{eq:common-rigidity-product}
        (\widetilde X,\widetilde\omega)
        \cong
        (\PP^r,a\omega_{\mathrm{FS}})\times(Y,\omega_{\mathrm{RF}}),
\end{equation}
where \(\omega_{FS}\) is the Fubini--Study metric and \(\omega_{RF}\) is Ricci-flat.
\end{proposition}

\begin{proof}
If \(m=0\), then \(Z\) is a point and the assumption on the general fibre gives \(X\cong\PP^r\). Let
\([\omega]=aH\), where \(H=c_1\bigl(\mathcal O_{\PP^r}(1)\bigr)\). Then \eqref{eq:sharp-relation-rigidity} implies that \(S(\omega)\) is constant. Since \([\omega]\) is proportional to \(c_1(X)\), \(\omega\) is K\"ahler--Einstein. By \cite{BandoMabuchi87}, \((X,\omega)\cong (\PP^r,a\omega_{\mathrm{FS}})\).

We may therefore assume \(m>0\) in the remainder of the proof. Choose a smooth compact K\"ahler model \(\nu:\widehat Z\to Z\) and a graph resolution
\[
        \begin{array}{ccc}
        \widehat X & \xrightarrow{\widehat\pi} & \widehat Z \\
        \mu\downarrow & & \downarrow\nu \\
        X & \dashrightarrow & Z.
        \end{array}
\]
By \cite[Theorem~1.10]{BierstoneMilman97}, we may take \(\mu\) to be a finite composition of blow-ups along smooth centres.  After shrinking to a connected Zariski open subset \(\widehat Z^\circ\subset\widehat Z\), the map \(\widehat\pi\) is a proper holomorphic submersion, \(\mu\) is biholomorphic on a neighbourhood of \(\widehat\pi^{-1}(\widehat Z^\circ)\), and \(\widehat F_z\cong\PP^r\), for \(z\in\widehat Z^\circ\).

Normalize \(H=c_1(\mathcal O_{\PP^r}(1))\) by \(H^r=1\).  The K\"ahler class of the fibres is locally constant in this smooth family, and every line in \(\PP^r\) is Poincar\'e dual to \(H^{r-1}\).  Hence for \(z\in\widehat Z^\circ\), we have \([\mu^*\omega|_{\widehat F_z}]=aH\).

By Theorem~\ref{thm:iterated-approximation}, there are K\"ahler metrics \(\omega_\ell\) on \(\widehat X\) such that
\begin{align*}
 \varepsilon_\ell:=\|S(\omega_\ell)-\mu^*S(\omega)\|_{C^0(\widehat X)}
 \longrightarrow0,\qquad
 \omega_\ell\longrightarrow\mu^*\omega
 \quad\text{in }C^\infty_{\mathrm{loc}}\left(\widehat X\setminus\operatorname{Exc}(\mu)\right).
\end{align*}
The classes of exceptional divisors restrict trivially to the regular fibres, hence \([\omega_\ell|_{\widehat F_z}]=aH\).

Since \(\widehat Z\) is not uniruled, Ou's characterization \cite[Theorem~1.1]{Ou2025} implies that \(K_{\widehat Z}\) is pseudo-effective.  Let
\[
        T=-\Ric(\eta)+\sqrt{-1}\partial\bar\partial\psi\ge0,
        \qquad \psi\le0,
\]
be the positive current representing \(2\pi c_1(K_{\widehat Z})\) used in Section~\ref{sec:level-set-method}. Then Proposition~\ref{prop:psef-fibrewise-scalar-final} and the fibrewise Chern--Weil identity give
\begin{align*}
 \min_{\widehat X}S(\omega_\ell)a^r \int_{\widehat Z^\circ}e^\psi\eta^m
 \le
 \int_{\widehat Z^\circ}e^\psi
 \left(\int_{\widehat F_z}S(\omega_\ell)
 \omega_{\ell,z}^r\right)\eta^m
 \le 2\pi r(r+1)a^{r-1}\int_{\widehat Z^\circ}e^\psi\eta^m.
\end{align*}
Subtracting the first term from the middle term therefore yields
\begin{align}
 &0\le\int_{\widehat Z^\circ}e^\psi
 \left[\int_{\widehat F_z}
 \bigl(S(\omega_\ell)-\min_{\widehat X}S(\omega_\ell)\bigr)
 \omega_{\ell,z}^r\right]\eta^m
 \longrightarrow0.                                      \label{eq:first-rigidity-defect}
\end{align}
Since \(e^\psi>0\) almost everywhere, the local smooth convergence in
Theorem~\ref{thm:iterated-approximation}, together with
\eqref{eq:first-rigidity-defect}, implies
\begin{equation}\label{eq:scalar-constant-rigidity}
        S(\omega)=\min_XS(\omega)
        =\frac{2\pi r(r+1)}a:=\widehat S.
\end{equation}

We next justify the equality case for the singular base metric. Choose the same regularizations as in
the proof of Proposition \ref{prop:psef-fibrewise-scalar-final}:
\[
        T_j=-\Ric(\eta)+\sqrt{-1}\partial\bar\partial\psi_j
        \ge-\delta_j\eta,
        \qquad
        \delta_j\downarrow0,
\]
where \(\psi_j\le0\) has analytic singularities and \(\psi_j\to\psi\) in \(L^1\) and almost everywhere.  Write
\[
        u_{\ell,j}
        =e^{\psi_j\circ\widehat\pi}J_{\widehat\pi,\omega_\ell},
        \qquad
        \widetilde T_j=T_j+\delta_j\eta\ge0.
\]
Define
\[
\begin{aligned}
        A_{\ell,j}
        :=
        \int_{\widehat Z^\circ}e^{\psi_j}
        \left(\int_{\widehat F_z}
        S(\omega_\ell)\,\omega_{\ell,z}^r\right)\eta^m,\qquad
        B_{\ell,j}
        :=
        \int_{\widehat Z^\circ}e^{\psi_j}
        \left(\int_{\widehat F_z}
        S(\omega_{\ell,z})\,\omega_{\ell,z}^r\right)\eta^m.
\end{aligned}
\]
Since \([\omega_\ell|_{\widehat F_z}]=aH\) and \( c_1(\widehat F_z)=(r+1)H\), the fibrewise Chern--Weil identity gives
\[
\begin{aligned}
        B_{\ell,j}
        =
        2\pi r(r+1)a^{r-1}\int_{\widehat Z^\circ}e^{\psi_j}\eta^m
        =
        \widehat S\,a^r\int_{\widehat Z^\circ}e^{\psi_j}\eta^m.
\end{aligned}
\]
On the other hand,
\[
        \int_{\widehat F_z}\omega_{\ell,z}^r=a^r,
        \qquad
        \mu^*S(\omega)\equiv\widehat S.
\]
Consequently,
\[
\begin{aligned}
        B_{\ell,j}-A_{\ell,j}
        &=
        \int_{\widehat Z^\circ}e^{\psi_j}
        \left(
        \int_{\widehat F_z}
        \bigl(\mu^*S(\omega)-S(\omega_\ell)\bigr)
        \omega_{\ell,z}^r
        \right)\eta^m.
\end{aligned}
\]
Since \(0\le e^{\psi_j}\le1\), it follows that
\begin{equation}\label{eq:uniform-AB-defect}
        |B_{\ell,j}-A_{\ell,j}|
        \le C\|S(\omega_\ell)-\mu^*S(\omega)\|_{C^0(\widehat X)},
\end{equation}
where \(C\) is independent of \(j\). 

By the same argument used to derive \eqref{eq:pre-coarea-weighted-final}, and by \eqref{eq:coarea-final} and \eqref{eq:vertical-trace-fibre-scalar-final}, we then have
\begin{align}
 &\frac12\int_{\widehat X}
 |\nabla^{\ell,j}s_{\widehat\pi}|^2\,\omega_\ell^{m+r}
 +\int_{\widehat X}u_{\ell,j}
 \tr_{\omega_\ell}\widehat\pi^*\widetilde T_j\,
 \omega_\ell^{m+r} \notag\\
 &\qquad\le
 \frac{(m+r)!}{m!r!}(B_{\ell,j}-A_{\ell,j})
 +\delta_j\int_{\widehat X}u_{\ell,j}
 \tr_{\omega_\ell}\widehat\pi^*\eta\,\omega_\ell^{m+r}.
                                                        \label{eq:rigidity-energy-estimate}
\end{align}
Since \(0\le u_{\ell,j}\le J_{\widehat\pi,\omega_\ell}\), \eqref{eq:rigidity-energy-estimate} and \eqref{eq:uniform-AB-defect} imply
\[
\begin{aligned}
 0\le{}&
 \frac12\int_{\widehat X}
 |\nabla^{\ell,j}s_{\widehat\pi}|^2\,\omega_\ell^{m+r}
 +\int_{\widehat X}u_{\ell,j}
 \tr_{\omega_\ell}\widehat\pi^*\widetilde T_j\,
 \omega_\ell^{m+r}\\
 \le{}&
 C\varepsilon_\ell+\delta_j \int_{\widehat X}
J_{\widehat\pi,\omega_\ell}
\tr_{\omega_\ell}\widehat\pi^*\eta\,\omega_\ell^{m+r} \\
={}&C\varepsilon_\ell+\delta_jC_\ell.
\end{aligned}
\]
We now choose a diagonal sequence.  For each \(\ell\), the number \(C_\ell\) is finite, while \(\delta_j\to0\) as \(j\to\infty\). Hence we may choose \(j(\ell)\ge\ell\) such that \(\delta_{j(\ell)}C_\ell\le\ell^{-1}\). Consequently,
\begin{align}
 &\int_{\widehat X}
 |\nabla^{\ell,j(\ell)}s_{\widehat\pi}|^2\,\omega_\ell^{m+r}
 +\int_{\widehat X}u_{\ell,j(\ell)}
 \tr_{\omega_\ell}\widehat\pi^*
 \widetilde T_{j(\ell)}\,\omega_\ell^{m+r}
 \longrightarrow0.                                    \label{eq:energy-to-zero}
\end{align}

Let \(K\Subset\widehat\pi^{-1}(\widehat Z^\circ)\). Since both terms in \eqref{eq:energy-to-zero} are nonnegative and \(\sqrt{u_{\ell,j}}=|s_{\widehat\pi}|_{\ell,j}\), by Kato's inequality, we have
\[
       \int_K
        \left|d\sqrt{u_{\ell,j(\ell)}}\right|_{\omega_\ell}^2
        \omega_\ell^{m+r}\le \int_{\widehat X}
        |\nabla^{\ell,j(\ell)}s_{\widehat\pi}|^2
        \,\omega_\ell^{m+r}
        \longrightarrow0.
\]
By Theorem~\ref{thm:iterated-approximation}, \(\omega_\ell\to\mu^*\omega\) in \(C^\infty(K)\).  Hence the corresponding norms and volume forms are uniformly equivalent on \(K\).  Consequently,
\begin{equation}\label{eq:Kato-local-rigidity}
        \int_K
        \left|d\sqrt{u_{\ell,j(\ell)}}\right|_{\mu^*\omega}^2
        (\mu^*\omega)^{m+r}
        \longrightarrow0.
\end{equation}

Since \(\widehat\pi\) is fixed, the same local smooth convergence implies
\[
        J_{\widehat\pi,\omega_\ell}
        \longrightarrow J_{\widehat\pi,\mu^*\omega}
        \quad\text{in }C^\infty(K).
\]
Moreover, \(j(\ell)\to\infty\), so \(e^{\psi_{j(\ell)}\circ\widehat\pi/2}\to e^{\psi\circ\widehat\pi/2}\) almost everywhere on \(K\). Because \(0\le e^{\psi_j/2}\le1\), dominated convergence yields
\[
        \sqrt{u_{\ell,j(\ell)}}
        \longrightarrow
        e^{\psi\circ\widehat\pi/2}
        \sqrt{J_{\widehat\pi,\mu^*\omega}}
        \quad\text{in }L^2(K).
\]
Together with \eqref{eq:Kato-local-rigidity}, the weak closedness of the
distributional differential shows that
\[
        d\left(
        e^{\psi\circ\widehat\pi/2}
        \sqrt{J_{\widehat\pi,\mu^*\omega}}
        \right)=0
\]
weakly on \(\widehat\pi^{-1}(\widehat Z^\circ)\).  This set is connected,
because \(\widehat Z^\circ\) is connected and the fibres are connected.
Consequently, for some constant \(c>0\),
\begin{equation}\label{eq:constant-weighted-Jacobian}
        e^{\psi\circ\widehat\pi}
        J_{\widehat\pi,\mu^*\omega}=c^2
        \quad\text{almost everywhere on }
        \widehat\pi^{-1}(\widehat Z^\circ).
\end{equation}
Here \(c>0\) because \(\psi\) is finite almost everywhere and \(J_{\widehat\pi,\mu^*\omega}>0\) on the regular locus.  Thus \(\psi\circ\widehat\pi=\log c^2-\log J_{\widehat\pi,\mu^*\omega}\) almost everywhere.
Since \(\log c^2-\log J_{\widehat\pi,\mu^*\omega}\) is smooth, we conclude that
\[
        T=-\Ric(\eta)+\sqrt{-1}\partial\bar\partial\psi
\]
is a smooth semipositive form on \(\widehat Z^\circ\), and \eqref{eq:constant-weighted-Jacobian} holds there pointwise.

By \eqref{eq:weighted-bochner-scalar-final}, \eqref{eq:scalar-constant-rigidity} and \eqref{eq:constant-weighted-Jacobian}, we obtain
\begin{equation}\label{eq:limiting-smooth-Bochner}
        0=|\nabla^\psi s_{\widehat\pi}|^2
        +c^2\left(
        \widehat S-\tr_V\Ric(\mu^*\omega)
        +\tr_{\mu^*\omega}\widehat\pi^*T
        \right)
\end{equation}
on \(\widehat\pi^{-1}(\widehat Z^\circ)\).  On each regular fibre,
\[
\begin{aligned}
 \int_{\widehat F_z}
 \bigl(\widehat S-\tr_V\Ric(\mu^*\omega)\bigr)
 (\mu^*\omega)_z^r
 =
 \widehat S a^r-2\pi r(r+1)a^{r-1}=0.
\end{aligned}
\]
Integrating \eqref{eq:limiting-smooth-Bochner} over \(\widehat F_z\) therefore gives 
\begin{equation}\label{eq:parallel-and-flat-base}
        \nabla^\psi s_{\widehat\pi}=0
        \quad\text{on }
        \widehat\pi^{-1}(\widehat Z^\circ),
        \qquad
        T=0
        \quad\text{on }\widehat Z^\circ.
\end{equation}

Let \(X^\circ:=\mu\bigl(\widehat\pi^{-1}(\widehat Z^\circ)\bigr)\subset X\). The map \(\mu\) identifies \(\widehat\pi^{-1}(\widehat Z^\circ)\) with the Zariski open subset \(X^\circ\). Under this identification, define \(V_x^\circ:=\left\{v\in T_x^{1,0}X:\iota_v s_{\widehat\pi}=0\right\}\). Since \(s_{\widehat\pi}\) is decomposable and nonzero on the regular locus, \(V^\circ\) is a rank-\(r\) holomorphic subbundle.  Moreover, \(\nabla^\psi s_{\widehat\pi}=0\) implies that \(V^\circ\) is parallel. Let \(P^\circ: T^{1,0}X|_{X^\circ}\rightarrow V^\circ\) be the \(\omega\)-orthogonal projection.  Since \(V^\circ\) is parallel, \(P^\circ\) satisfies
\[
        (P^\circ)^2=P^\circ,
        \qquad
        (P^\circ)^{*\omega}=P^\circ,
        \qquad
        \nabla^\omega P^\circ=0.
\]
In particular, \(P^\circ\) is holomorphic and has operator norm one. Hence the coefficients of \(P^\circ\) in every local holomorphic frame are locally bounded near the analytic subset \(X\setminus X^\circ\). The Riemann extension theorem gives a unique holomorphic endomorphism \(P\in H^0\bigl(X,\operatorname{End}(T^{1,0}X)\bigr)\) extending \(P^\circ\). In particular, \(P\) has constant rank \(r\). Consequently,
\[
        T^{1,0}X
        =
        \operatorname{Im}P\oplus\ker P
\]
is an orthogonal decomposition into parallel holomorphic subbundles. The de Rham theorem then gives the holomorphic isometric splitting \eqref{eq:common-rigidity-product} on the universal cover. The vertical factor is \(\PP^r\).  By \eqref{eq:limiting-smooth-Bochner} and \eqref{eq:parallel-and-flat-base}, its induced metric has constant scalar curvature in the class \(aH\).  Since \(c_1(\PP^r)=(r+1)H\), this metric is K\"ahler--Einstein; the Bando--Mabuchi uniqueness theorem \cite{BandoMabuchi87} identifies it with \(a\omega_{FS}\) up to an automorphism.  Finally, on the product, \eqref{eq:constant-weighted-Jacobian} identifies the Ricci form of the horizontal factor with \(-T=0\).  Thus its induced metric \(\omega_{RF}\) is Ricci-flat.
\end{proof}

\bibliographystyle{alpha}
\bibliography{wpref}

\bigskip

\noindent
\textsc{Zehao Sha}\\
Institute for Mathematics and Fundamental Physics, Hefei, 230088, China\\
\textit{Email address:} \texttt{zhsha@imfp.org.cn}

\bigskip

\noindent
\textsc{Jian Wang}\\
State Key Laboratory of Mathematical Sciences,
Academy of Mathematics and Systems Science, Chinese Academy of Sciences,
Beijing 100190, China\\
\textit{Email address:} \texttt{jian.wang.4@amss.ac.cn}

\end{document}